\documentclass[10pt, a4paper]{amsart}  
\usepackage{amsmath,amssymb,amsfonts,graphicx,color,subfig,hyperref}

\newcommand{\R}{\mathbb{R}}

\newcommand{\Z}{\mathbb{Z}}
\newcommand{\C}{\mathbb{C}}

\newcommand{\Ccal}{\mathcal{C}}
\newcommand{\Fcal}{\mathcal{F}}
\newcommand{\Kcal}{\mathcal{K}}
\newcommand{\Pcal}{\mathcal{P}}

\newcommand{\tsp}{{\sf T}}

\newcommand{\finv}{\Fcal^{-1}}

\newcommand{\floor}[1]{\lfloor #1 \rfloor}

\DeclareMathOperator{\vol}{vol}

\DeclareMathOperator{\ort}{O}

\def\spheresdp/{A}
\def\spacesdp/{B}

\newtheorem{defin}{Definition}[section]

\newtheorem{theorem}[defin]{Theorem}

\newtheorem{lemma}[defin]{Lemma}

\newcommand{\defi}[1]{\textit{#1}}

\begin{document}

\title{Upper bounds for packings of spheres of~several~radii}

\author{David de Laat}

\address{D.~de Laat, Delft Institute of Applied Mathematics, Delft
  University of Technology, P.O. Box 5031, 2600 GA Delft, The
  Netherlands} 
\email{mail@daviddelaat.nl}

\author{Fernando M\'ario de Oliveira Filho}

\address{F.M.~de Oliveira Filho, Institut f\"ur Mathematik,
Freie Universit\"at Berlin, Arnim\-allee 2, 14195 Berlin, Germany}

\email{fmario@mi.fu-berlin.de}

\author{Frank Vallentin} 

\address{F.~Vallentin, Delft Institute of Applied Mathematics, Delft
  University of Technology, P.O. Box 5031, 2600 GA Delft, The
  Netherlands}

\email{f.vallentin@tudelft.nl}

\thanks{The first and third author were supported by Vidi grant
  639.032.917 from the Netherlands Organization for Scientific
  Research (NWO). The second author was supported by Rubicon grant
  680-50-1014 from the Netherlands Organization for Scientific
  Research (NWO)}



\date{June 12, 2012}

\begin{abstract}
  We give theorems that can be used to upper bound the densities of
  packings of different spherical caps in the unit sphere and of
  translates of different convex bodies in Euclidean space. These
  theorems extend the linear programming bounds for packings of
  spherical caps and of convex bodies through the use of semidefinite
  programming. We perform explicit computations, obtaining new
  bounds for packings of spherical caps of two different sizes and
  for binary sphere packings. We also slightly improve bounds for
  the classical problem of packing identical spheres.
\end{abstract}

\maketitle

\markboth{D. de Laat, F.M.~de Oliveira Filho, and
F.~Vallentin}{Upper bounds for packings of spheres of several radii}

\section{Introduction}

How densely can one pack given objects into a given container?
Problems of this sort, generally called \textit{packing problems}, are
fundamental problems in geometric optimization.

An important example having a rich history is the sphere packing
problem. Here one tries to place equal-sized spheres with pairwise
disjoint interiors into\linebreak $n$-dimensional Euclidean space while
maximizing the fraction of covered space. In two dimensions the best
packing is given by placing open disks centered at the points of the
hexagonal lattice. In three dimensions, the statement that the best
sphere packing has density~$\pi/\sqrt{18} = 0.7404\ldots$ was known as
Kepler's conjecture; it was proved by Hales~\cite{Hales2005} in 1998
by means of a computer-assisted proof. 

Currently, one of the best methods for obtaining upper bounds for the
density of sphere packings is due to Cohn and
Elkies~\cite{CohnElkies2003}. In 2003 they used linear programming to
obtain the best known upper bounds for the densities of sphere
packings in dimensions~$4, \ldots, 36$. They almost closed the gap
between lower and upper bounds in dimensions~$8$ and~$24$. Their
method is the noncompact version of the linear programming method of
Delsarte, Goethals, and Seidel~\cite{DelsarteGoethalsSeidel1977} for
upper-bounding the densities of packings of spherical caps on the unit
sphere.

From a physical point of view, packings of spheres of different sizes
are relevant as they can be used to model chemical mixtures which
consist of multiple atoms or, more generally, to model the structure
of composite material. For more about technological applications of
these kind of systems of polydisperse, totally impenetrable spheres we
refer to Torquato \cite[Chapter 6]{Torquato2001}. In recent work, Hopkins,
Jiao, Stillinger, and
Torquato~\cite{HopkinsJiaoStillingerTorquato2011,
  HopkinsStillingerTorquato2012} presented lower bounds for the
densities of packings of spheres of two different sizes, also called
\defi{binary sphere packings}.

In coding theory, packings of spheres of different sizes are important
in the design of error-correcting codes which can be used for
unequal error protection. Masnick and
Wolf~\cite{MasnickWolf1967} were the first who considered codes with
this property.

In this paper we extend the linear programming method of Cohn and
Elkies to obtain new upper bounds for the densities of multiple-size
sphere packings. We also extend the linear programming method of
Delsarte, Goethals, and Seidel to obtain new upper bounds for the
densities of multiple-size spherical cap packings.

We perform explicit calculations for binary packings in both cases
using semidefinite, instead of linear, programming. In particular we
complement the constructive lower bounds of Hopkins, Jiao, Stillinger,
and Torquato by non-constructive upper bounds. Insights gained from
our computational approach are then used to improve known upper bounds
for the densities of monodisperse sphere packings in dimensions~$4$,
\dots~$9$, except~$8$. The bounds we present improve on the best-known
bounds due to Cohn and Elkies~\cite{CohnElkies2003}.

\subsection{Methods and theorems}
\label{sec:methods and theorems}

We model the packing problems using tools from combinatorial
optimization. All possible positions of the objects which we can use
for the packing are vertices of a graph and we draw edges between two
vertices whenever the two corresponding objects cannot be
simultaneously present in the packing because they overlap in their
interiors. Now every independent set in this conflict graph gives a
valid packing and vice versa. To determine the density of the packing
we use vertex weights since we want to distinguish between ``small''
and ``big'' objects. For finite graphs it is known that the weighted
independence number can be upper bounded by the weighted theta
number. Our theorems for packings of spherical caps and spheres are
infinite-dimensional analogues of this result.

Let~$G = (V, E)$ be a finite graph. A set~$I\subseteq V$ is 
\emph{independent} if no two vertices in~$I$ are adjacent. Given a
weight function~$w \colon V \to \R_{\geq 0}$, the \emph{weighted
independence number} of~$G$ is the maximum weight of an independent
set, i.e.,
\[
\alpha_w(G) = \max\biggl\{\sum_{x \in I} w(x) : \text{$I \subseteq V$
  is independent}\biggr\}.
\]
Finding $\alpha_w(G)$ is an NP-hard problem.

Gr\"otschel, Lov\'asz, and
Schrijver~\cite{GroetschelLovaszSchrijver1981} defined a graph
parameter that gives an upper bound for~$\alpha_w$ and which can be
computed efficiently by semidefinite optimization. It can be presented
in many different, yet equivalent ways, but the one convenient for us is
\[
\begin{array}{rll}
\vartheta'_w(G) = \min&M\\
&K - (w^{1/2}) (w^{1/2})^{\sf T}&\text{is positive semidefinite},\\
&K(x,x) \leq M&\text{for all $x \in V$},\\
&K(x,y) \leq 0&\text{for all $\{x,y\} \not\in E$ where $x \neq y$},\\
&\text{$M \in \R$, $K \in \R^{V \times V}$ is symmetric.}\span
\end{array}
\]

Here we give a proof of the fact that $\vartheta'_w(G)$ upper
bounds $\alpha_w(G)$. In a sense, after discarding the analytical arguments
in the proofs of Theorems~\ref{thm:sphere} and~\ref{thm:space}, we are
left with this simple proof.

\begin{theorem}
\label{thm:theta}
  For any finite graph $G = (V, E)$ with weight function $w\colon V \to
  \R_{\geq 0}$ we have $\alpha_w(G) \leq \vartheta'_w(G)$.
\end{theorem}

\begin{proof}
  Let $I \subseteq V$ be an independent set of nonzero weight and let
  $K \in \R^{V \times V}$, $M \in \R$ be a feasible solution of
  $\vartheta'_w(G)$. Consider the sum
\begin{equation*}
\label{eq:thetasum}
\sum_{x,y \in I} w(x)^{1/2} w(y)^{1/2} K(x,y).
\end{equation*}

This sum is at least 
\[
\sum_{x,y \in I} w(x)^{1/2} w(y)^{1/2}
w(x)^{1/2} w(y)^{1/2} = \biggl(\sum_{x \in I} w(x)\biggr)^2
\]
because $K - (w^{1/2}) (w^{1/2})^{\sf T}$ is positive semidefinite. 

The sum is also at most
\[
\sum_{x \in I} w(x) K(x,x) \leq M \sum_{x \in I} w(x)
\]
because $K(x,x) \leq M$ and because $K(x,y) \leq 0$ whenever $x \neq
y$ as $I$ forms an independent set. Now combining both inequalities
proves the theorem.
\end{proof}

\subsection*{Multiple-size spherical cap packings}

We first consider packings of spherical caps of several radii on the
unit sphere~$S^{n-1} = \{\, x \in \R^n : x \cdot x = 1\, \}$. The
spherical cap with angle~$\alpha \in [0,\pi]$ and center $x \in
S^{n-1}$ is given by
\[
C(x,\alpha) = \{\,y \in S^{n-1} : x \cdot y \geq \cos \alpha\,\}.
\]
Its normalized volume equals
\[
w(\alpha) = \frac{\omega_{n-1}(S^{n-2})}{\omega_n(S^{n-1})} \int_{\cos \alpha}^1 (1-u^2)^{(n-3)/2}\, du,
\]
where~$\omega_n(S^{n-1}) = (2\pi^{n/2})/\Gamma(n/2)$ is the surface
area of the unit sphere. Two spherical caps $C(x_1,\alpha_1)$ and
$C(x_2,\alpha_2)$ intersect in their topological interiors if and only
if the inner product of $x_1$ and $x_2$ lies in the
interval~$(\cos(\alpha_1 + \alpha_2), 1]$. Conversely we have
\[
C(x_1,\alpha_1)^\circ \cap C(x_2,\alpha_2)^\circ = \emptyset\quad
\iff\quad
x_1 \cdot x_2 \leq \cos(\alpha_1 + \alpha_2).
\]
A \emph{packing} of spherical caps with angles~$\alpha_1$,
\dots,~$\alpha_N$ is a union of any number of spherical caps with
these angles and pairwise-disjoint interiors. The
density of the packing is the sum of the normalized volumes of the
constituting spherical caps. 

The optimal packing density is given by the weighted independence number
of the \emph{spherical cap packing graph}. This is the graph with
vertex set $S^{n-1} \times \{1, \ldots, N\}$, where a vertex $(x, i)$
has weight $w(\alpha_i)$, and where two distinct vertices~$(x, i)$
and~$(y, j)$ are adjacent if~$\cos(\alpha_i + \alpha_j) <
x \cdot y$.

In Section~\ref{sec:spherical cap packings} we will extend the
weighted theta prime number to the spherical cap packing graph. There
we will also derive Theorem~\ref{thm:sphere} below, which gives upper
bounds for the densities of packings of spherical caps. We will show
that the sharpest bound given by this theorem is in fact equal to the
theta prime number.

In what follows we denote by~$P_k^n$ the Jacobi
polynomial~$P^{((n-3)/2, (n-3)/2)}_k$ of degree~$k$, normalized so
that~$P_k^n(1) = 1$.

\begin{theorem}
\label{thm:sphere}
Let~$\alpha_1$, \dots,~$\alpha_N \in (0, \pi]$ be angles and for~$i$,
$j = 1$, \dots,~$N$ and~$k \geq 0$ let~$f_{ij,k}$ be real numbers such
that~$f_{ij,k} = f_{ji,k}$ and~$\sum_{k=0}^\infty |f_{ij,k}| < \infty$
for all~$i$, $j$. Write
\begin{equation}
\label{eq:spheref}
f_{ij}(u) = \sum_{k=0}^\infty f_{ij, k} P^{n}_k(u).
\end{equation}

Suppose the functions $f_{ij}$ satisfy the following conditions:
\begin{enumerate}
\item[(i)] \label{cond:sphere1} $\bigl(f_{ij, 0} - w(\alpha_i)^{1/2}
  w(\alpha_j)^{1/2}\bigr)_{i, j=1}^N$ is positive semidefinite;
\item[(ii)] \label{cond:sphere2} $\bigl(f_{ij, k}\bigr)_{i,j=1}^N$ is
  positive semidefinite for $k \geq 1$;
\item[(iii)] \label{cond:sphere3} $f_{ij}(u) \leq 0$ whenever~$-1 \leq u \leq \cos(\alpha_i
  + \alpha_j)$.
\end{enumerate}
Then the density of every packing of spherical caps with
angles~$\alpha_1$, \dots,~$\alpha_N$ on the unit sphere~$S^{n-1}$ is
at most~$\max\{\,f_{ii}(1) : \text{$i = 1$, \dots,~$N$}\,\}$.
\end{theorem}

When~$N = 1$, Theorem~\ref{thm:sphere} reduces to the linear
programming bound for spherical cap packings of Delsarte, Goethals,
and Seidel \cite{DelsarteGoethalsSeidel1977}. In Section~\ref{sec:unit
  sphere computations} we use semidefinite programming instead of
linear programming to perform explicit computations for $N = 2$.

\subsection*{Translational packings of bodies and multiple-size sphere
  packings}

We now deal with packings of spheres with several radii in
$\R^n$. Theorem~\ref{thm:space} presented below can be used to find
upper bounds for the densities of such packings. In fact, it is more
general and can be applied to packings of translates of different
convex bodies.

Let~$\Kcal_1$, \dots,~$\Kcal_N$ be convex bodies in~$\R^n$. A translational
\emph{packing} of $\Kcal_1$,~\dots,~$\Kcal_N$ is
a union of translations of these bodies in which any two copies have
disjoint interiors. The \defi{density} of a packing is the fraction of
space covered by it. There are different ways to formalize this
definition, and questions appear as to whether every packing has a
density and so on. We postpone further discussion on this matter until
Section~\ref{sec:sphere packings} where we give a proof of
Theorem~\ref{thm:space}.

Our theorem can be seen as an analogue of the weighted theta prime
number~$\vartheta_w'$ for the infinite graph~$G$ whose vertex set
is~$\R^n \times \{1, \ldots, N\}$ and in which vertices~$(x, i)$
and~$(y, j)$ are adjacent if~$x + \Kcal_i$ and~$y + \Kcal_j$ have
disjoint interiors. The weight function we consider assigns
weight~$\vol \Kcal_i$ to vertex~$(x, i) \in \R^n \times \{ 1, \ldots,
N \}$. We will say more about this interpretation in
Section~\ref{sec:sphere packings}.

For the statement of the theorem we need some basic facts from
harmonic analysis. Let~$f\colon \R^n \to \C$ be an~$L^1$
function. For~$u \in \R^n$, the Fourier transform of~$f$ at~$u$ is
\[
\hat{f}(u) = \int_{\R^n} f(x) e^{-2\pi i u \cdot x}\, dx.
\]
We say that function~$f$ is a \defi{Schwartz function} (also called a
\defi{rapidly-decreasing function}) if it is infinitely
differentiable, and if any derivative of~$f$, multiplied by any power
of the variables~$x_1$, \dots,~$x_n$, is a bounded function. The
Fourier transform of a Schwartz function is a
Schwartz function, too. A Schwartz function can be recovered from its
Fourier transform by means of the \defi{inversion formula}:
\[
f(x) = \int_{\R^n} \hat{f}(u) e^{2\pi i u \cdot x}\, du
\]
for all~$x \in \R^n$.

\begin{theorem}
\label{thm:space}
Let~$\Kcal_1$, \dots,~$\Kcal_N$ be convex bodies in~$\R^n$ and
let~$f\colon \R^n \to \R^{N \times N}$ be a matrix-valued function
whose every component~$f_{ij}$ is a Schwartz function. Suppose~$f$
satisfies the following conditions:
\begin{enumerate}

\item[(i)] the matrix~$\bigl(\hat{f}_{ij}(0) - (\vol \Kcal_i)^{1/2}
  (\vol \Kcal_j)^{1/2}\bigr)_{i,j=1}^N$ is positive semidefinite;

\item[(ii)] the matrix of Fourier
  transforms~$\bigl(\hat{f}_{ij}(u)\bigr)_{i,j=1}^N$ is
  positive semidefinite for every~$u \in \R^n \setminus \{0\}$;

\item[(iii)] $f_{ij}(x) \leq 0$ whenever~$\Kcal_i^{\circ} \cap (x
  + \Kcal_j^{\circ}) = \emptyset$.
\end{enumerate}
Then the density of any packing of translates of~$\Kcal_1$,
\dots,~$\Kcal_N$ in the Euclidean space~$\R^n$ is at most~$\max\{\,
f_{ii}(0) : \text{$i = 1$, \dots, $N$}\,\}$.
\end{theorem}

We give a proof of this theorem in Section~\ref{sec:sphere
  packings}. When~$N = 1$ and when the convex body~$\Kcal_1$ is
centrally symmetric (an assumption which is in fact not needed) then this theorem
reduces to the linear programming method of Cohn and
Elkies~\cite{CohnElkies2003}.

We apply this theorem to obtain upper bounds for the densities of
binary sphere packings, as we discuss in
Section~\ref{ssec:computations space}.

\subsection{Computational results for binary spherical cap packings}
\label{ssec:computations sphere}

We applied Theorem~\ref{thm:sphere} to compute upper bounds for the
densities of binary spherical cap packings. The results we obtained
are summarized in the plots of Figure~\ref{fig:binary cap bounds}.

\begin{figure}[tb]
\centering
\subfloat[$n=3$]{\includegraphics{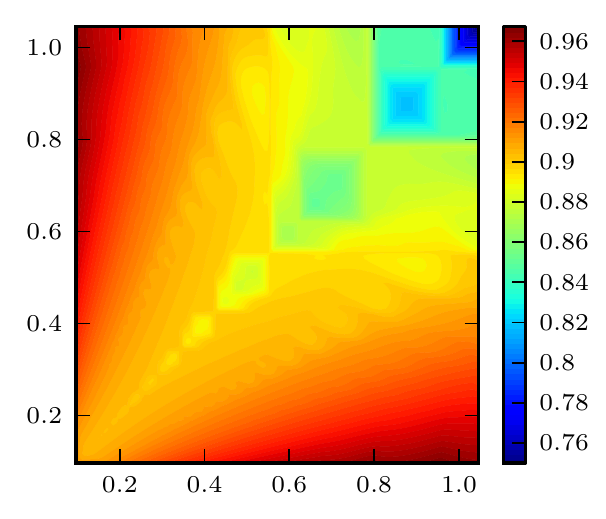}}
\subfloat[SDP bound/geometric bound for $n=3$]{\label{fig:binary cap
    bounds b}\includegraphics{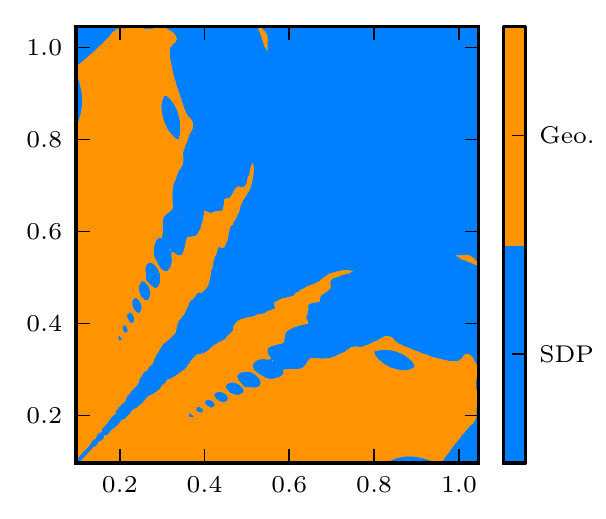}}

\subfloat[$n=4$]{\includegraphics{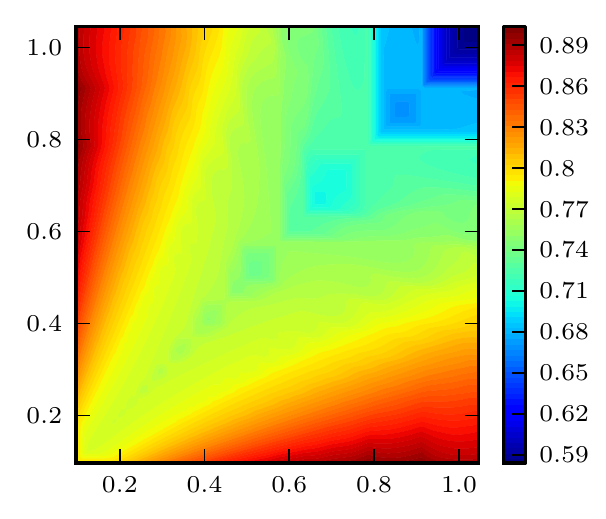}}
\subfloat[$n=5$]{\includegraphics{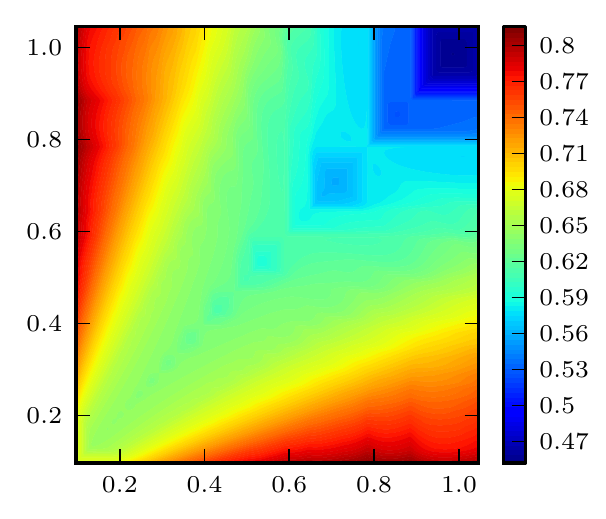}}
\caption{Upper bounds on the packing density for $N = 2$. The
  horizontal and vertical axes carry the spherical cap angle; the
  colors indicate the density, or in the case of plot (b) whether the
  SDP bound or the geometric bound is sharper.}
\label{fig:binary cap bounds}
\end{figure}

For $n=3$, Florian~\cite{Florian2001, Florian2007} provides a
geometric upper bound for the density of a spherical cap packing. He
shows that the density of a packing on $S^2$ of spherical caps with
angles $\alpha_1, \ldots, \alpha_N \in (0, \pi / 3]$ is at most
\[
\max_{1 \leq i \leq j \leq k \leq N} D(\alpha_i, \alpha_j, \alpha_k),
\]
where $D(\alpha_i, \alpha_j, \alpha_k)$ is defined as follows. Let
$\mathcal{T}$ be a spherical triangle in $S^2$ such that if we center
the spherical caps with angles $\alpha_i$, $\alpha_j$, and $\alpha_k$
at the vertices of $\mathcal{T}$, then the caps intersect pairwise at
their boundaries. The number $D(\alpha_i, \alpha_j, \alpha_k)$ is then
defined as the fraction of the area of $\mathcal{T}$ covered by the
caps.

In Figure~\ref{fig:binary cap bounds b} we see that for $N = 2$ it
depends on the angles whether the geometric or the semidefinite
programming bound is sharper. In particular we see that near the
diagonal the semidefinite programming bound is at least as good as the
geometric bound; see also Figure~\ref{fig:mono cap bounds a}.

We can construct natural multiple-size spherical cap packings by
taking the incircles of the faces of spherical Archimedean tilings. A
sequence of binary packings is for instance obtained by taking the
incircles of the prism tilings. These are the Archimedean tilings with
vertex figure $(4, 4, k)$ for $k \geq 3$ (although strictly speaking
for $k = 4$ this is a spherical Platonic tiling). The question then is
whether the packing associated with the $k$-prism has maximal density
among all packings with the same cap angles $\pi / k$ and $\pi / 2 -
\pi / k$, that is, whether the packing is \emph{maximal}. The packing
for $k = 3$ is not maximal while the one for $k = 4$ trivially is,
since here there is only one cap size, and adding a $9$th cap yields a
density greater than~$1$.

Heppes and Kert\'esz~\cite{HeppesKertesz1997} showed that the
configurations for $k \geq 6$ are maximal, and the remaining case $k =
5$ was later shown to maximal by Florian and
Heppes~\cite{FlorianHeppes1999}. Florian~\cite{Florian2001} showed
that the geometric bound given above is in fact sharp for the cases
where $k \geq 6$, and for the $k = 5$ case it is not sharp but still
good enough to prove maximality (notice that given a finite number of
cap angles, the set of obtainable densities is finite).

Now we illustrate that Theorem~\ref{thm:sphere} gives a
sharp bound for the density of the packing associated to the
$5$-prism, thus giving a simple proof of its maximality. The theorem
also provides a sharp bound for $n = 4$ but whether it can provide
sharp bounds for the cases $n \geq 6$ we do not know at the
moment. The numerical results are not decisive.

We shall exhibit functions
\[
f_{ij}(u) = \sum_{k=0}^4 f_{ij,k} P^n_k(u)
\]
which satisfy the conditions of Theorem~\ref{thm:sphere} with
$f_{11}(1) = 5 w(\alpha_1) + 2 w(\alpha_2)$ where
\[
\alpha_1 = \frac{\pi}{5},\;
\alpha_2 = \frac{3\pi}{10},\;
w(\alpha_1) = \frac{1}{2}\left(1 - \cos \frac{\pi}{5}\right),\;
w(\alpha_2) = \frac{1}{2}\left(1 - \cos \frac{3\pi}{10}\right).
\]
By complementary slackness of semidefinite optimization the
coefficients $f_{ij,k}$  have to satisfy the following linear conditions:
\[
0 = f_{11}\left(\cos\frac{2\pi}{5}\right) =
f_{11}\left(\cos\frac{4\pi}{5}\right) =
f'_{11}\left(\cos\frac{2\pi}{5}\right) = f_{12}(0) = f_{22}(-1);
\]
the product 
\[
\begin{pmatrix}
f_{11,0} & f_{12,0}\\
f_{12,0} & f_{22,0}
\end{pmatrix}
\begin{pmatrix}
25 w(\alpha_1) & 10 \sqrt{w(\alpha_1)w(\alpha_2)}\\
10 \sqrt{w(\alpha_1)w(\alpha_2)} & 4 w(\alpha_2)
\end{pmatrix}
\]
equals
\[
\begin{pmatrix}
25 w(\alpha_1)^2 + 10 w(\alpha_1)w(\alpha_2) &
\sqrt{w(\alpha_1)w(\alpha_2)} (10 w(\alpha_1) + 4 w(\alpha_2))\\
\sqrt{w(\alpha_1)w(\alpha_2)} (25 w(\alpha_1) +
10 w(\alpha_2))
& 10 w(\alpha_1)w(\alpha_2)  + 4 w(\alpha_2)^2
\end{pmatrix};
\]
for $k = 1, \ldots, 4$ the product of the two matrices 
$
\begin{pmatrix}
f_{11,k} & f_{12,k}\\
f_{12,k} & f_{22,k}
\end{pmatrix}
$
and
\[
\begin{pmatrix}
w(\alpha_1)(5P_k(1) + 10P_k(\cos\frac{2\pi}{5}) + 10
P_k(\cos \frac{2\pi}{4})) & \sqrt{w(\alpha_1)w(\alpha_2)} 10
P_k(0)\\
\sqrt{w(\alpha_1)w(\alpha_2)} 10 P_k(0) & w(\alpha_2) (2P_k(1) + 2P_k(-1))
\end{pmatrix}
\]
equals zero. This linear system together with the additional assumptions
\[
0 = f_{11}(-1) = f_{12}\left(-\frac{95}{100}\right) = f'_{12}\left(-\frac{95}{100}\right)
\]
has a one-dimensional space of solutions from which it is easy to
select one which fulfills all requirements of Theorem~\ref{thm:sphere}.

For the remaining $13$ Archimedean solids in dimension $n=3$ we are
only able to show maximality of the packing associated to the
truncated octahedron, the Archimedean solid with vertex figure $(6, 6,
5)$. Its density is $0.9056\ldots$, the geometric bound shows that
the density is at most $0.9088\ldots$, and using the semidefinite
program we get $0.9079\ldots$ as an upper bound. The first packing
with caps of angles $\arcsin(1 / 3)$ and $\arcsin(1 / \sqrt{3})$ which
would be denser is obtained by taking $19$ of the smaller caps and $4$
of the bigger caps, and has density $0.9103\ldots$ The upper bounds
show however that it is not possible to obtain this dense a packing,
thus showing that the truncated octahedron packing is maximal.

\begin{figure}[htb]
\centering
\subfloat[$n=3$]{\includegraphics{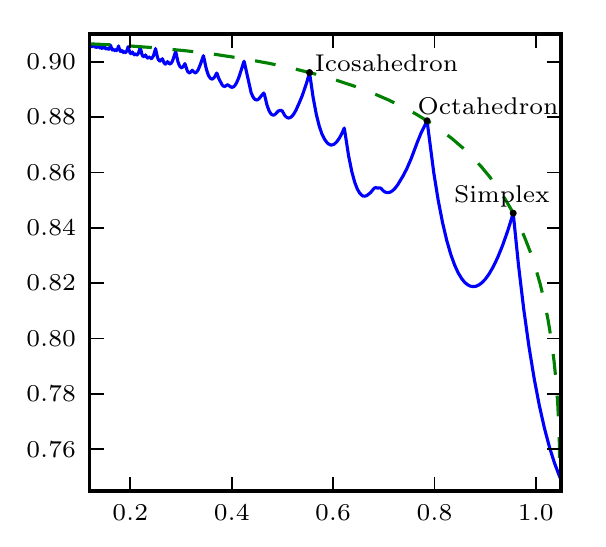}\label{fig:mono cap bounds a}}\hspace{0.0001\textwidth}
\subfloat[$n=4$]{\includegraphics{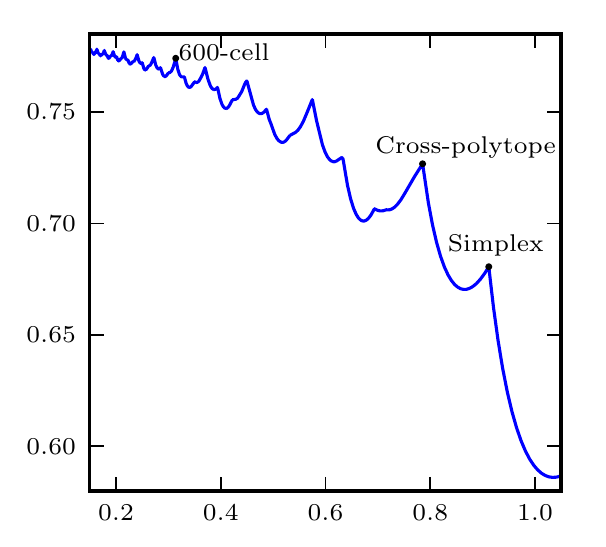}}

\subfloat[$n=5$]{\includegraphics{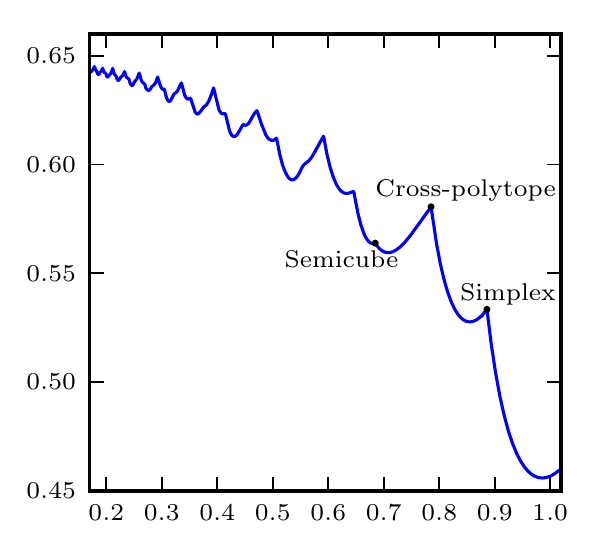}}
\caption{Upper bounds on the packing density for $N = 1$, the
  horizontal axis carries the spherical cap angle and the vertical
  axis the packing density.}
\label{fig:single sized cap bounds}
\end{figure}

We also used our programs to plot the upper bounds for $N = 1$, the
classical linear programming bound of Delsarte, Goethals, and
Seidel~\cite{DelsarteGoethalsSeidel1977}, for dimensions $n = 3$, $4$,
and~$5$ in Figure~\ref{fig:single sized cap bounds}. To the best of
our knowledge these kinds of plots were not made before and they seem
to reveal interesting properties of the bound. For better orientation
we show in the plots the packings where the linear programming bound
is sharp (cf.~Levenshtein~\cite{Levenshtein1998}; Cohn and
Kumar~\cite{CohnKumar2007} proved the much stronger statement that these
packings provide point configurations which are universally optimal). The
dotted line in the plot for $n = 3$ is the geometric bound, and since
we know that both the geometric (cf.~Florian~\cite{Florian2001}) and
the semidefinite programming bounds are sharp for the given
configurations, we know that at these peaks the bounds meet.

An interesting feature of the upper bound seems to be that it has some
periodic behavior. Indeed, the numerical results suggest that for $n =
3$, the two bounds in fact meet infinitely often as the angle
decreases, and that between any two of these meeting points the
semidefinite programming bound has a similar shape. Although in higher
dimensions we do not have a geometric bound, the semidefinite
programming bound seems to admit the same kind of periodic behavior.

\subsection{Computational results for binary sphere packings}
\label{ssec:computations space}

We applied Theorem~\ref{thm:space} to compute upper bounds for the
densities of binary sphere packings. The results we obtained are
summarized in the plot of Figure~\ref{fig:space-plots}, where we show
bounds computed for dimensions~$2$, \dots,~$5$. A detailed account of
our approach is given in Section~\ref{sec:Euclidean space
  computations}.  We now quickly discuss the bounds presented in
Figure~\ref{fig:space-plots}.

\begin{figure}[t]
\begin{center}
\includegraphics{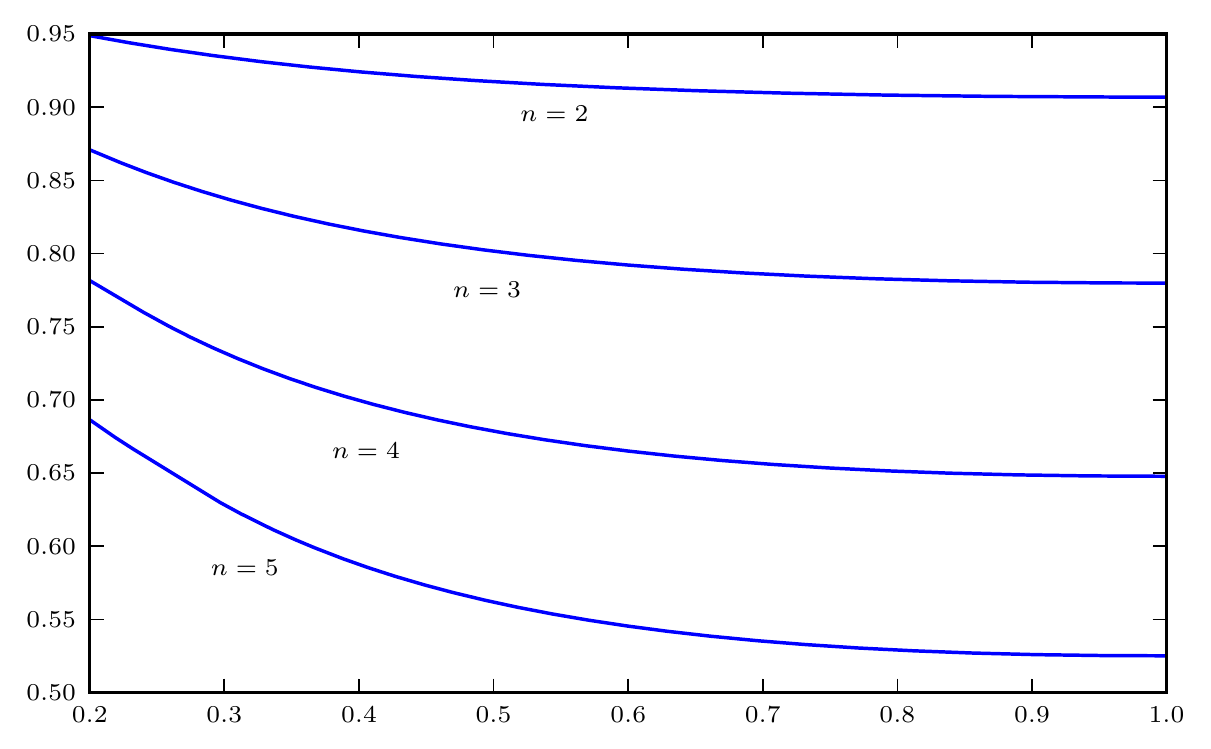}
\end{center}
\caption{The horizontal axis carries the ratio between
  the radii of the small and the large spheres. The vertical axis
  carries our upper bound. Our bounds for dimensions $2$, \dots,~$5$ are shown
  together.}
\label{fig:space-plots}
\end{figure}

\smallbreak

\noindent
\textsl{Dimension 2.}\enspace Only in dimension $2$ have binary sphere
(i.e., circle) packings been studied in depth. We refer to the
introduction in the paper of Heppes~\cite{Heppes2003} which surveys
the known results about binary circle packings in the plane.

Currently one of the best-known upper bounds for the maximum density of a binary
circle packing is due to Florian~\cite{Florian1960}. Florian's bound
states that a packing of circles in which the ratio between the radii of
the smallest and largest circles is~$r$ has density at most
\[
\frac{\pi r^2 + 2(1 - r^2) \arcsin (r / (1 + r))}{2 r \sqrt{2r + 1}},
\]
and that this bound is achieved exactly for~$r = 1$ (i.e., for
classical circle packings) and for~$r = 0$ in the limit.

The question arises of which bound is better, our bound or Florian's
bound. From our experiments, it seems that our bound is worse than
Florian's bound, at least for~$r < 1$. For instance, for~$r = 1/2$ we
obtain the upper bound~$0.9174426\ldots$, whereas Florian's bound
is~$0.9158118\ldots$ Whether this really means that the bound of
Theorem~\ref{thm:space} is worse than Florian's bound, or just that
the computational approach of Section~\ref{sec:Euclidean space computations} 
is too restrictive to attain his bound, we do not know.

It is interesting to note that for~$r = 1$, that is, for packings of
circles of one size, our bound clearly coincides with the one of Cohn
and Elkies~\cite{CohnElkies2003}. This bound seems to be equal to
$\pi/\sqrt{12}$, but no proof of this is known.
\smallbreak

\noindent
\textsl{Dimension 3.}\enspace Much less is known in dimension $3$. In
fact we do not know about other attempts to find upper bounds for the
densities of binary sphere packings in dimensions~$3$ and higher.

Let us compare our upper bound with the lower bound by Hopkins, Jiao,
Stillinger, and Torquato~\cite{HopkinsJiaoStillingerTorquato2011}. The
record holder for $r \geq 0.2$ in terms of highest density occurs for $r =
0.224744\ldots$ and its density is $0.824539\ldots$ Our computations
show that there cannot be a packing with this $r$ having density more
than $0.8617125\ldots$, so this leaves a margin of~$5\%$. 

Another interesting case is $r = \sqrt{2} - 1 = 0.414\ldots$ Here the
best-known lower bound of $0.793\ldots$ comes from the NaCl-alloy. The
large spheres are centered at a face centered cubic lattice and the
small spheres are centered at a translated copy of the face centered
cubic lattice so that they form a jammed packing. Our upper bound for
$r = \sqrt{2} - 1$ is $0.813\ldots$, less than~$3\%$ away from the
lower bound. Therefore, we believe that proving optimality of the
NaCl-alloy might be doable.
\smallbreak

\noindent
\textsl{Dimension 4 and beyond.}\enspace In higher dimensions even
less is known about binary sphere packings. We observed from
Figure~\ref{fig:space-plots} that it seems that the upper bound is
decreasing: as the radius of the small sphere increases from~$0.2$
to~$1$, the bound seems to decrease. This suggests that the bound
given by Theorem~\ref{thm:space} is decreasing in this sense, but we
do not know a proof of this.

We also do not know the limit behavior of our bound when~$r$
approaches~$0$. Due to numerical instabilities we could not perform
numerical calculations in this regime of~$r$.

\subsection{Improving the Cohn-Elkies bounds}
\label{sec:improve-ce}

We now present a theorem that can be used to find better upper bounds
for the densities of monodisperse sphere packings than those provided by Cohn and
Elkies~\cite{CohnElkies2003}; our theorem is a strengthening of theirs.

Fix~$\varepsilon > 0$. Given a packing of spheres of radius~$1/2$, we
consider its \defi{$\varepsilon$-tangency graph}, a graph whose
vertices are the spheres in the packing, and in which two vertices are
adjacent if the distance between the centers of the respective spheres
lies in the interval~$[1, 1 + \varepsilon)$.  

Let~$M(\varepsilon)$ be the least upper bound on the average degree of
the $\varepsilon$-tangency graph of any sphere packing. Our theorem is
the following:

\begin{theorem}
\label{thm:improve-ce}
Take~$0 = \varepsilon_0 < \varepsilon_1 < \cdots < \varepsilon_m$ and
let~$f\colon \R^n \to \R$ be a Schwartz function such that
\begin{enumerate}
\item[(i)] $\hat{f}(0) \geq \vol B$, where~$B$ is the ball of
  radius~$1/2$;

\item[(ii)] $\hat{f}(u) \geq 0$ for all~$u \in \R^n \setminus \{ 0
  \}$;

\item[(iii)] $f(x) \leq 0$ whenever~$\|x\| \geq 1 + \varepsilon_m$;

\item[(iv)] $f(x) \leq \eta_k$ whenever~$\|x\| \in [1 +
  \varepsilon_{k-1}, 1 + \varepsilon_k)$ with~$\eta_k \geq 0$, for~$k
  = 1$, \dots,~$m$.
\end{enumerate}
Then the density of a sphere packing is at most the optimal
value of the following linear programming problem in variables~$A_1$,
\dots,~$A_m$:
\begin{equation}
\label{eq:ce-lp}
\begin{array}{rll}
\max&f(0) + \eta_1 A_1 + \cdots + \eta_m A_m\\
&A_1 + \cdots + A_k \leq U(\varepsilon_k)&\text{for~$k = 1$,
  \dots,~$m$},\\
&A_i \geq 0&\text{for~$i = 1$, \dots,~$m$},
\end{array}
\end{equation}
where~$U(\varepsilon_k) \geq M(\varepsilon_k)$ for~$k = 1$,
\dots,~$m$.
\end{theorem}

In Section~\ref{sec:improve-ce-detail} we give a proof of
Theorem~\ref{thm:improve-ce} and show how to compute upper bounds
for~$M(\varepsilon)$ using the semidefinite programming bounds of
Bachoc and Vallentin~\cite{BachocV2008} for the sizes of spherical
codes. There we also show how to use semidefinite programming and the
same ideas we employ in the computations for binary sphere packings
(cf.~Section~\ref{sec:Euclidean space computations}) to compute better
upper bounds for the densities of sphere packings.

In Table~\ref{tab:improve-ce} we show the upper bounds obtained
through our application of Theorem~\ref{thm:improve-ce}. To better
compare our bounds with those of Cohn and Elkies, on
Table~\ref{tab:improve-ce} we show bounds for the center density of a
packing, the \defi{center density} of a packing of unit spheres being
equal to~$\Delta / \vol B$, where~$\Delta$ is the density of the
packing, and~$B$ is a unit ball.

We omit dimension~$8$ because for this dimension it is already
believed that the Cohn-Elkies bound is itself optimal, and therefore
as is to be expected we did not manage to obtain any improvement over
their bound. We also note that the bounds by Cohn and Elkies are the
best known upper bounds in all other dimensions shown.

\begin{table}[htb]
\begin{tabular}{cccc}
{\sl Dimension}&{\sl Lower bound}&{\sl Cohn-Elkies bound}&{\sl New
  upper bound}\\[2pt]
\vrule height11.5pt width0pt $4$&0.12500&0.13126&0.130587\\
$5$&0.08839&0.09975&0.099408\\
$6$&0.07217&0.08084&0.080618\\
$7$&0.06250&0.06933&0.069193\\
$9$&0.04419&0.05900&0.058951
\end{tabular}
\bigskip

\caption{For each dimension we show the
  best lower bound known, the bound by Cohn and Elkies~\cite{CohnElkies2003},
  and the upper bound coming from Theorem~\ref{thm:improve-ce}.}
\label{tab:improve-ce}
\end{table}

In dimension~$3$ the Cohn-Elkies bound is $0.18616$ whereas the
optimal sphere packing has center density $0.17678$. We can improve
the Cohn-Elkies bound to $0.184559$ which is also better than the
upper bound~$0.1847$ due to Rogers~\cite{Rogers58}.

\section{Multiple-size spherical cap packings}
\label{sec:spherical cap packings}
In this section we prove Theorem~\ref{thm:sphere} and discuss its
relation to an extension of the weighted theta prime number for the
spherical cap packing graph.

\subsection{Proof of Theorem~\ref{thm:sphere}}

Let~$x_1$, \dots,~$x_m \in S^{n-1}$ and~$r\colon \{ 1, \ldots, m \}
\to \{ 1, \ldots, N \}$ be such that
\[
\bigcup_{i=1}^m C(x_i, \alpha_{r(i)})
\]
is a packing of spherical caps on~$S^{n-1}$.

Consider the sum
\begin{equation}
\label{eq:sphere-sum}
\sum_{i, j = 1}^m w(\alpha_{r(i)})^{1/2} w(\alpha_{r(j)})^{1/2} f_{r(i)r(j)}(x_i \cdot x_j).
\end{equation}
By expanding~$f_{r(i)r(j)}(x_i \cdot x_j)$ according to \eqref{eq:spheref} this sum is equal to
\[
\sum_{k=0}^\infty \sum_{i,j=1}^m w(\alpha_{r(i)})^{1/2}
w(\alpha_{r(j)})^{1/2} f_{r(i) r(j), k} P_k^n(x_i \cdot x_j).
\]
By the addition formula (cf.~e.g.~Section~9.6 of Andrews, Askey, and
Roy~\cite{AndrewsAskeyRoy1999}) for the Jacobi polynomials $P_k^n$ the
matrix~$\bigl(P_k^n(x_i \cdot x_j)\bigr)_{i,j=1}^m$ is positive
semidefinite. From condition~\hyperref[cond:sphere2]{(ii)} of the
theorem, we also know that the matrix~$\bigl(f_{r(i) r(j),
  k}\bigr)_{i,j=1}^m$ is positive semidefinite for~$k \geq 1$. So the
inner sum above is nonnegative for $k \geq 1$. If we then consider
only the summand for~$k = 0$ we see that~\eqref{eq:sphere-sum} is
at least
\begin{equation}
\label{eq:sphere-lb}
\sum_{i,j=1}^m w(\alpha_{r(i)})^{1/2} w(\alpha_{r(j)})^{1/2} f_{r(i) r(j), 0} P_0^n(x_i \cdot x_j) \geq \biggl(\sum_{i=1}^m  w(\alpha_i)\biggr)^2,
\end{equation}
where the inequality follows from condition~\hyperref[cond:sphere1]{(i)} of the theorem.

Now, notice that whenever~$i \neq j$, the caps~$C(x_i, \alpha_{r(i)})$
and~$C(x_j, \alpha_{r(j)})$ have disjoint 
interiors. Condition~\hyperref[cond:sphere3]{(iii)} then implies that~$f_{r(i) r(j)}(x_i \cdot
x_j) \leq 0$. So we see that~\eqref{eq:sphere-sum} is at most
\begin{equation}
\label{eq:sphere-ub}
\sum_{i=1}^m w(\alpha_i) f_{r(i) r(i)}(1) \leq 
\max\{\,f_{ii}(1) : \text{$i = 1$, \dots,~$N$}\,\} \sum_{i=1}^m w(\alpha_i).
\end{equation}

So~\eqref{eq:sphere-sum} is at least~\eqref{eq:sphere-lb} and
at most~\eqref{eq:sphere-ub}, yielding
\[
\sum_{i=1}^m w(\alpha_i) \leq \max\{\,f_{ii}(1) : \text{$i = 1$,
  \dots,~$N$}\,\}. \qquad  \qed
\]

\subsection{Theorem~\ref{thm:sphere} and the Lov\'asz theta number}
\label{ssec:generalization}

We now briefly discuss a generalization of~$\vartheta'_w$ to infinite
graphs and its relation to the bound of
Theorem~\ref{thm:sphere}. Similar ideas were developed by Bachoc,
Nebe, Oliveira, and Vallentin~\cite{BachocNebeOliveiraVallentin2009}.

Let $G = (V, E)$ be a graph, where~$V$ is a compact space, and
let~$w\colon V \to \R_{\geq 0}$ be a continuous weight function. An
element in the space $\mathcal{C}(V \times V)$ of real-valued
continuous functions over~$V \times V$ is called a \defi{kernel}. A
kernel~$K$ is \defi{symmetric} if $K(x, y) = K(y, x)$ for all $x, y
\in V$. It is \defi{positive} if it is symmetric and if for any $N \in
\mathbb{N}$ and for any $x_1$,~\dots,~$x_N \in V$, the matrix~$\bigl(K(x_i, x_j)\bigr)_{i, j =
  1}^N$ is positive semidefinite. The \defi{weighted theta prime number}
of~$G$ is defined as
\begin{equation}
\label{eq:generalized theta}
\begin{array}{rll}
\vartheta'_w(G) = \inf&M\\
&K - w^{1/2} \otimes (w^{1/2})^*&\text{is a positive kernel},\\
&K(x,x) \leq M&\text{for all $x \in V$},\\
&K(x,y) \leq 0&\text{for all $\{x,y\} \not\in E$ where $x \neq
  y$},\\
&\text{$M \in \R$, $K \in \mathcal{C}(V \times V)$ is symmetric}.\span\\
\end{array}
\end{equation}
One may show, mimicking the proof of Theorem~\ref{thm:theta},
that~$\vartheta'_w(G) \geq \alpha_w(G)$.

Let $G = (V, E)$ be the spherical cap packing graph as defined in
Section~\ref{sec:methods and theorems}. We will use the symmetry of
this graph to show that~\eqref{eq:generalized theta} gives the
sharpest bound obtainable by Theorem~\ref{thm:sphere}.

The orthogonal group~$\ort(n)$ acts on~$S^{n-1}$, and this
defines the action of~$\ort(n)$ on the vertex set $V = S^{n-1} \times
\{1, \ldots, N\}$ by $A(x, i) = (Ax, i)$ for~$A \in \ort(n)$. The
group average of a kernel~$K \in \mathcal{C}(V \times V)$ is given by
\[
\overline{K}((x, i), (y, j)) = \int_{\ort(n)} K(A(x, i), A(y, j))\, d\mu(A), 
\]
where $\mu$ is the Haar measure on~$\ort(n)$ normalized so
that~$\mu(\ort(n)) = 1$. If $(K, M)$ is feasible for
\eqref{eq:generalized theta}, then $(\overline{K}, M)$ is feasible
too. This follows since for each $A \in \ort(n)$, a point $(x, i)$ has
the same weight as $A(x, i)$, and two points $(x, i)$ and $(y, j)$ are
adjacent if and only if $A(x, i)$ and $A(y, j)$ are adjacent. Since
$(K, M)$ and $(\overline{K}, M)$ have the same objective value~$M$, and
since $\overline{K}$ is invariant under the action of $\ort(n)$, we
may restrict to $\ort(n)$-invariant kernels (i.e., kernels~$K$ such
that~$K(Au,Av) = K(u,v)$ for all~$A \in \ort(n)$ and~$u$, $v \in V$)
in finding the infimum of~\eqref{eq:generalized theta}.

Schoenberg~\cite{Schoenberg1942} showed that a symmetric kernel~$K \in
\mathcal{C}(S^{n-1} \times S^{n-1})$ is positive and $\ort(n)$-invariant
if and only if it lies in the cone spanned by the kernels $(x, y)
\mapsto P_k^n(x \cdot y)$. We will use the following generalization for kernels over~$V \times V$.

\begin{theorem}\label{lemma:pos inv kernel}
  A symmetric kernel $K \in \mathcal{C}(V \times V)$, with $V =
  S^{n-1} \times \{1, \ldots, N\}$, is positive and
  $\ort(n)$-invariant if and only if
\begin{equation}\label{invpos}
K((x, i), (y, j)) = f_{ij}(x \cdot y)
\end{equation}
with
\[
f_{ij}(u) =  \sum_{k=0}^\infty f_{ij,k} P_k^n(u),
\]
where $\bigl(f_{ij,k}\bigr)_{i, j = 1}^N$ is positive semidefinite
for all $k \geq 0$ and~$\sum_{k=0}^\infty |f_{ij,k}| < \infty$ for
all~$i$, $j = 1$, \dots,~$N$, implying in particular that we have
uniform convergence above.
\end{theorem}

Before we prove the theorem we apply it to simplify
problem~\eqref{eq:generalized theta}.  If~$K$ is an
$\ort(n)$-invariant feasible solution of~\eqref{eq:generalized theta},
then~$K - w^{1/2} \otimes (w^{1/2})^*$ is a positive
$\ort(n)$-invariant kernel, and hence can be written in the
form~\eqref{invpos}. Using in addition that $P_0^n = 1$,
problem~\eqref{eq:generalized theta} reduces to
\[
\begin{array}{rll}
\vartheta'_w(G) = \inf&M\\
&f_{ii}(0) + w(\alpha_i) \leq M&\text{for all $1 \leq i \leq N$},\\
&f_{ij}(u) + (w(\alpha_i) w(\alpha_j))^{1/2} \leq 0&\text{when
  $-1 \leq u \leq \cos(\alpha_i + \alpha_j)$},\\
&\text{$M \in \R$ and $\bigl(f_{ij,k}\bigr)_{i, j = 1}^N$ positive semidefinite
  for all $k \geq 0$}.\span
\end{array}
\]
By substituting $f_{ij,0} - (w(\alpha_i) w(\alpha_j))^{1/2}$ for
$f_{ij,0}$ we see that the solution to this problem indeed equals the
sharpest bound given by Theorem~\ref{thm:sphere}.

\begin{proof}[Proof of Theorem~\ref{lemma:pos inv kernel}]\label{proof:pos inv kernel}
  If we endow the space $\mathcal{C}(S^{n-1})$ of real-valued
  continuous function on the unit sphere $S^{n-1}$ with the usual
  $L^2$ inner product, then for~$f$, $g \in \Ccal(V)$,
\[
\langle f, g \rangle = \sum_{i=1}^N \int_{S^{n-1}} f(x, i) g(x, i)\, d\omega(x)
\]
gives an inner product on $\mathcal{C}(V)$. The space $\mathcal{C}(S^{n-1})$
decomposes orthogonally as
\[
\mathcal{C}(S^{n-1}) =  \bigoplus_{k=0}^\infty H_k,
\]
where $H_k$ is the space of homogeneous harmonic polynomials of degree $k$ restricted
to $S^{n-1}$. With
\[
H_{k, i} = \{\, f \in \mathcal{C}(V) : \text{there is a $g \in H_k$
  such that $f(\cdot, j) = \delta_{ij} g(\cdot)$}\,\},
\]
it follows that $\mathcal{C}(V)$ decomposes orthogonally as
\[
\mathcal{C}(V) = \bigoplus_{k=0}^\infty \bigoplus_{i=1}^N H_{k, i}.
\]
Given the action of $\ort(n)$ on $V$, we have the natural unitary
representation on $\mathcal{C}(V)$ given by $(Af)(x,i) = f(A^{-1}x,
i)$ for $A \in \ort(n)$ and $f \in \mathcal{C}(V)$.  It follows that each
space $H_{k, i}$ is $\ort(n)$-irreducible and that two spaces $H_{k, i}$
and $H_{k', i'}$ are $\ort(n)$-equivalent if and only if $k = k'$. Let
\[
\{\, e_{k, i, l} : \text{$k \geq 0$, $1 \leq i \leq N$, and $1 \leq l \leq h_k$} \,\}
\]
be a complete orthonormal system of
$\mathcal{C}(V)$ such that $e_{k, i, 1},$ \dots,~$e_{k, i, h_k}$ is a
basis of $H_{k, i}$. By Bochner's characterization \cite{Bochner1941},
a kernel $K \in \mathcal{C}(V \times V)$ is positive and
$\ort(n)$-invariant if and only if
\begin{equation}
\label{eq:K}
K((x, i), (y, j)) = \sum_{k = 0}^\infty \sum_{i', j' = 1}^N f_{ij,k} \sum_{l=1}^{h_k} e_{k, i', l}(x, i) e_{k, j', l}(y, j), 
\end{equation}
where each $\bigl(f_{ij,k}\bigr)_{i, j = 1}^N$ is positive
semidefinite and~$\sum_{k=0}^\infty |f_{ij,k}| < \infty$ for all~$i$,
$j$.

By the addition formula (cf.~Chapter~9.6 of Andrews, Askey, and
Roy~\cite{AndrewsAskeyRoy1999}) we have
\[
\sum_{l=1}^{h_k} e_{k, l}(x) e_{k, l}(y)
=
\frac{h_k}{\omega_n(S^{n-1})}
 P_k^n(x \cdot y)
\]
for any orthonormal basis $e_{k, 1}, \ldots, e_{k, h_k}$ of $H_k$. It follows that 
\[
\sum_{l=1}^{h_k}e_{k, i', l}(x, i) e_{k, j', l}(y, j) = \delta_{ii'}
\delta_{jj'} \frac{h_k}{\omega_n(S^{n-1})} P_k^n(x \cdot y),
\]
and substituting this into~\eqref{eq:K} completes the proof.
\end{proof}

Bochner's characterization for the kernel~$K$, which we used above,
usually assumes that the spaces under consideration are homogeneous,
so that the decompositions into isotypic irreducible spaces are
guaranteed to be finite. This finiteness is then used to conclude
uniform convergence. Since the action of $\ort(n)$ on $V$ is not
transitive, we do not immediately have this guarantee. We can still
use the characterization, however, since irreducible subspaces of
$\mathcal{C}(V)$ have finite multiplicity.

\section{Translational packings of bodies and multiple-size sphere  packings}
\label{sec:sphere packings}

Before giving a proof of Theorem~\ref{thm:space} we quickly present
some technical considerations regarding density. Here we follow
closely Appendix~A of Cohn and Elkies~\cite{CohnElkies2003}.

Let~$\Kcal_1$, \dots,~$\Kcal_N$ be convex bodies and~$\Pcal$ be a
packing of translated copies of $\Kcal_1$,~\dots,~$\Kcal_N$, that
is,~$\Pcal$ is a union of translated copies of the bodies, any two
copies having disjoint interiors. We say that the \defi{density}
of~$\Pcal$ is~$\Delta$ if for all~$p \in \R^n$ we have
\[
\Delta = \lim_{r \to \infty} \frac{\vol(B(p, r) \cap \Pcal)}{\vol B(p,
  r)},
\]
where~$B(p, r)$ is the ball of radius~$r$ centered at~$p$. Not every
packing has a density, but every packing has an \defi{upper density}
given by
\[
\limsup_{r \to \infty} \sup_{p \in \R^n} \frac{\vol(B(p, r) \cap \Pcal)}{\vol B(p,
  r)}.
\]

We say that a packing~$\Pcal$ is \defi{periodic} if there is a lattice~$L
\subseteq \R^n$ that leaves~$\Pcal$ invariant, that is, which is such
that~$\Pcal = x + \Pcal$ for all~$x \in L$. In other words, a periodic
packing consists of some translated copies of the bodies~$\Kcal_1$,
\dots,~$\Kcal_N$ arranged inside the fundamental parallelotope of~$L$,
and this arrangement repeats itself at each copy of the fundamental
parallelotope translated by vectors of the lattice.

It is easy to see that a periodic packing has a density. This is
particularly interesting for us, since in computing upper bounds for
the maximum possible density of a packing we may restrict ourselves to
periodic packings, as it is known (and easy to see) that the supremum
of the upper densities of packings is also achieved by periodic
packings (cf.~Appendix~A in Cohn and Elkies~\cite{CohnElkies2003}).

To provide a proof of the theorem we need another fact from harmonic
analysis, the Poisson summation formula. Let~$f\colon \R^n \to \C$ be
a Schwartz function and~$L \subseteq \R^n$ be a lattice. The
\defi{Poisson summation formula} states that, for every~$x \in \R^n$,
\[
\sum_{v \in L} f(x + v) = \frac{1}{\vol(\R^n / L)} \sum_{u \in L^*}
\hat{f}(u) e^{2\pi i u \cdot x},
\]
where~$L^* = \{\, u \in \R^n : \text{$u \cdot x \in \Z$ for all~$x
  \in L$}\,\}$ is the \defi{dual lattice} of~$L$ and where $\vol(\R^n
/ L)$ is the volume of a fundamental domain of the lattice~$L$.

\begin{proof}[Proof of Theorem~\ref{thm:space}]
As observed above, we may restrict ourselves to periodic packings. Let
then~$L \subseteq \R^n$ be a lattice and~$x_1$, \dots,~$x_m
\in \R^n$ and~$r\colon \{ 1, \ldots, m \} \to \{ 1, \ldots, N \}$ be
such that
\[
\Pcal = \bigcup_{v \in L} \bigcup_{i=1}^m v + x_i + \Kcal_{r(i)}
\]
is a packing. This means that, whenever~$i \neq j$ or~$v \neq 0$,
bodies~$x_i + \Kcal_{r(i)}$ and~$v + x_j + \Kcal_{r(j)}$ have disjoint
interiors. This packing is periodic and therefore has a well-defined
density, which equals
\[
\frac{1}{\vol(\R^n / L)} \sum_{i=1}^m \vol \Kcal_{r(i)}.
\]

Consider the sum
\begin{equation}
\label{eq:space-sum}
\sum_{v \in L} \sum_{i,j = 1}^m (\vol \Kcal_{r(i)})^{1/2}
(\vol \Kcal_{r(j)})^{1/2} f_{r(i)r(j)}(v + x_j - x_i).
\end{equation}
Applying the Poisson summation formula we may
express~\eqref{eq:space-sum} in terms of Fourier transform
of~$f$, obtaining
\[
\frac{1}{\vol(\R^n / L)}\sum_{u \in L^*} \sum_{i,j = 1}^m
(\vol K_{r(i)})^{1/2}(\vol K_{r(j)})^{1/2} \hat{f}_{r(i)r(j)}(u) e^{2\pi i u
  \cdot (x_j - x_i)},
\]
where~$L^*$ is the dual lattice of~$L$.

Since~$f$ satisfies condition~(ii) of the theorem, 
matrix~$\bigl(\hat{f}_{r(i) r(j)}(u)\bigr)_{i,j=1}^m$ is positive
semidefinite for every~$u \in \R^n$. So the inner sum above is always
nonnegative. If we then consider only the summand for~$u = 0$, we see
that~\eqref{eq:space-sum} is at least
\begin{equation}
\label{eq:space-lower}
\begin{split}
&\frac{1}{\vol(\R^n / L)}\sum_{i,j=1}^m (\vol \Kcal_{r(i)})^{1/2}(\vol \Kcal_{r(j)})^{1/2}\hat{f}_{r(i)
  r(j)}(0)\\
&\qquad{}\geq \frac{1}{\vol(\R^n / L)}\sum_{i, j = 1}^m \vol
\Kcal_{r(i)} \vol \Kcal_{r(j)}\\
&\qquad{}= \frac{1}{\vol(\R^n / L)}\biggl(\sum_{i=1}^m \vol \Kcal_{r(i)}\biggr)^2,
\end{split}
\end{equation}
where the inequality comes from condition~(i) of the theorem.

Now, notice that whenever~$v \neq 0$ or~$i \neq j$ one
has~$f_{r(i)r(j)}(v + x_j - x_i) \leq 0$. Indeed, since~$\Pcal$ is a
packing, if~$v \neq 0$ or~$i \neq j$ then the bodies~$x_i +
\Kcal_{r(i)}$ and~$v + x_j + \Kcal_{r(j)}$ have disjoint
interiors. But then also~$\Kcal_{r(i)}$ and~$v + x_j - x_i +
\Kcal_{r(j)}$ have disjoint interiors, and then from~(iii) we see
that~$f_{r(i) r(j)}(v + x_j - x_i) \leq 0$.

From this observation we see immediately that~\eqref{eq:space-sum} is
at most
\begin{equation}
\label{eq:space-upper}
\sum_{i=1}^m \vol \Kcal_{r(i)} f_{r(i) r(i)}(0) \leq \max\{\,
f_{ii}(0) : \text{$i = 1$, \dots,~$N$}\,\} \sum_{i=1}^m
\vol \Kcal_{r(i)}.
\end{equation}

So~\eqref{eq:space-sum} is at least~\eqref{eq:space-lower} and
at most~\eqref{eq:space-upper}. Putting it all together we get
that
\[
\frac{1}{\vol(\R^n / L)} \sum_{i=1}^m \vol \Kcal_{r(i)} \leq \max\{\,
f_{ii}(0) : \text{$i = 1$, \dots,~$N$}\,\},
\]
proving the theorem.
\end{proof}

We mentioned in the beginning of the section that
Theorem~\ref{thm:space} is an analogue of the weighted theta prime
number for a certain infinite graph. The connection will become more
clear after we present a slightly more general version of
Theorem~\ref{thm:space}.

An~$L^\infty$ function~$f\colon \R^n \to \C^{N \times N}$ is said to
be of \defi{positive type} if~$f(x) = f(-x)^*$ for all~$x \in \R^n$
and for all~$L^1$ functions~$\rho\colon \R^n \to \C^N$ we have
\[
\int_{\R^n} \int_{\R^n} \rho(y)^* f(x - y) \rho(x)\, dx dy \geq 0.
\]

When~$N = 1$ we have the classical theory of functions of positive
type (see e.g.~the book by Folland~\cite{Folland1995} for
background). Many useful properties of such functions can
be extended to the matrix-valued case (that is, to the~$N > 1$ case)
via a simple observation: a function~$f\colon \R^n \to \C^{N \times
  N}$ is of positive type if and only if for all~$p \in \C^N$ the
function~$g_p\colon \R^n \to \C$ such that
\[
g_p(x) = p^* f(x) p
\]
is of positive type.

From this observation two useful classical characterizations of
functions of positive type can be extended to the matrix-valued
case. The first one is useful when dealing with continuous functions
of positive type. It states that a continuous and bounded
function~$f\colon \R^n \to \C^{N \times N}$ is of positive type if and
only if for every choice~$x_1$, \dots,~$x_m$ of finitely many points in~$\R^n$, the
block matrix~$\bigl(f(x_i - x_j)\bigr)_{i, j = 1}^m$ is positive
semidefinite. 

The second characterization is given in terms of the Fourier
transform. It states that an~$L^1$ function~$f\colon \R^n \to \C^{N
  \times N}$ is of positive type if and only if the
matrix~$\bigl(\hat{f}_{ij}(u)\bigr)_{i,j=1}^N$ is positive
semidefinite for all~$u \in \R^n$. So in the statement of
Theorem~\ref{thm:space}, for instance, one could replace condition~(i)
by the equivalent condition that~$f$ be a function of positive type.

When~$N = 1$, the previous two characterizations of functions of
positive type date back to Bochner~\cite{Bochner1932}.

With this we may give an alternative and more general version of
Theorem~\ref{thm:space}.

\begin{theorem}
\label{thm:space-general}
Let~$\Kcal_1$, \dots,~$\Kcal_N$ be convex bodies in~$\R^n$ and
let~$f\colon \R^n \to \R^{N \times N}$ be a continuous and~$L^1$
function. Suppose~$f$ satisfies the following conditions:
\begin{enumerate}
\item[(i)] the matrix~$\bigl(\hat{f}_{ij}(0) - (\vol \Kcal_i)^{1/2}
  (\vol \Kcal_j)^{1/2}\bigr)_{i,j=1}^N$ is positive semidefinite;

\item[(ii)] $f$ is of positive type;

\item[(iii)] $f_{ij}(x) \leq 0$ whenever~$\Kcal_i^\circ \cap (x
  + \Kcal_j^\circ) = \emptyset$.
\end{enumerate}
Then the density of every packing of translates of~$\Kcal_1$,
\dots,~$\Kcal_N$ in the Euclidean space~$\R^n$ is at most~$\max\{\,
f_{ii}(0) : \text{$i = 1$, \dots,~$N$}\,\}$.
\end{theorem}

Let~$V = \R^n \times \{ 1, \ldots, N \}$. Notice that the
kernel~$K\colon V \times V \to \R$ such that
\[
K((x, i), (y, j)) = f_{ij}(x - y),
\]
implicitly defined by the function~$f$, plays the same role as the
matrix~$K$ from the definition of the theta prime number
(cf.~Section~\ref{ssec:generalization}). For instance, this is a
positive kernel, since~$f$ is of positive type and hence for any~$L^1$
function~$\rho\colon V \to \R$ we have that
\[
\int_V \int_V K((x, i), (y, j)) \rho(x, i) \rho(y, j)\, d(x, i) d(y,
j) \geq 0.
\]
Theorem~\ref{thm:space-general} can then be seen as an analogue of the
weighted theta prime number for the packing graph with vertex set~$V$
that we consider.

When one reads through the proof of Theorem~\ref{thm:space}, the one
step that fails when~$f$ is~$L^1$ instead of Schwartz is the use of
the Poisson summation formula. Indeed, sum~\eqref{eq:space-sum} is not
anymore well-defined in such a situation. The summation formula also
holds, however, under somewhat different conditions that are just what
we need to make the proof go through. The proof of the following lemma
makes use of the well-known interpretation of the Poisson summation
formula as a trace formula, which for instance is explained by
Terras~\cite[Chapter 1.3]{Terras1985}.

\begin{lemma}
\label{lem:poisson}
Let~$f\colon \R^n \to \C^{N \times N}$ be a continuous function of
bounded support and positive type. Then for every lattice~$L \subseteq
\R^n$, every~$x \in \R^n$, and all~$i$, $j = 1$, \dots,~$N$ we have
\[
\sum_{v \in L} f_{ij}(x + v) = \frac{1}{\vol(\R^n / L)} \sum_{u \in
  L^*} \hat{f}_{ij}(u) e^{2\pi i u \cdot x}.
\]
\end{lemma}

\begin{proof}
Since each function~$f_{ij}$ is continuous and of bounded support, the
functions~$g_{ij}\colon \R^n / L \to \C$ such that
\[
g_{ij}(x) = \sum_{v \in L} f_{ij}(x + v)
\]
are continuous. Indeed, the sum above is well-defined, being in fact a
finite sum (since~$f_{ij}$ has bounded support), and
therefore~$g_{ij}$ can be seen locally as a sum of finitely many continuous
functions.

Let us now compute the Fourier transform of~$g_{ij}$. For~$u \in
L^*$ we have that
\[
\begin{split}
\hat{g}_{ij}(u)&= \int_{\R^n / L} g_{ij}(x) e^{-2\pi i u \cdot x}\, dx\\
&= \int_{\R^n / L} \sum_{v \in L} f_{ij}(x + v) e^{-2\pi i u \cdot x}\,
dx\\
&= \int_{\R^n} f_{ij}(x) e^{-2\pi i u \cdot x}\, dx\\
&= \hat{f}_{ij}(u).
\end{split}
\]
So we know that
\begin{equation}
\label{eq:gij-l2}
g_{ij}(x) = \frac{1}{\vol(\R^n / L)} \sum_{u \in L^*} \hat{f}_{ij}(u)
e^{2\pi i u \cdot x}
\end{equation}
in the sense of~$L^2$ convergence. Our goal is to prove that pointwise
convergence also holds above.

To this end we consider for~$i = 1$, \dots,~$N$ the kernel~$K_i\colon
(\R^n / L) \times (\R^n / L) \to \C$ such that
\[
K_i(x, y) = \sum_{v \in L} f_{ii}(v + x - y).
\]

Since each function~$f_{ii}$ is of bounded support and continuous,
each kernel~$K_i$ is continuous. Since for each~$i$ we have
that~$f_{ii}(x) = \overline{f_{ii}(-x)}$ for all~$x \in \R^n$
(since~$f$ is of positive type), each kernel~$K_i$ is
self-adjoint. Notice that the functions~$x \mapsto (\vol(\R^n /
L))^{-1/2} e^{2\pi i u \cdot x}$, for~$u \in L^*$, form a complete
orthonormal system of~$L^2(\R^n / L)$. Each such function is also an
eigenfunction of~$K_i$, with eigenvalue~$\hat{f}_{ii}(u)$. Indeed, we
have
\[
\begin{split}
&\int_{\R^n / L} K_i(x, y) (\vol(\R^n / L))^{-1/2} e^{2\pi i u \cdot
  y}\, dy\\
&\qquad{}= (\vol(\R^n / L))^{-1/2} \int_{\R^n / L} \sum_{v \in L} f_{ii}(v + x -
y) e^{2\pi i u \cdot y}\, dy\\
&\qquad{}= (\vol(\R^n / L))^{-1/2} \int_{\R^n} f_{ii}(x - y) e^{2\pi i u \cdot
  y}\, dy\\
&\qquad{}= (\vol(\R^n / L))^{-1/2} \int_{\R^n} f_{ii}(y) e^{2\pi i u \cdot (x
  - y)}\, dy\\
&\qquad{}= \hat{f}_{ii}(u) (\vol(\R^n / L))^{-1/2} e^{2\pi i u \cdot x}.
\end{split}
\]

Since~$f$ is of positive type, the matrices of Fourier
transforms~$\bigl(\hat{f}_{ij}(u)\bigr)_{i,j=1}^N$, for~$u \in \R^n$,
are all positive semidefinite. In particular this implies that the
Fourier transforms of~$f_{ii}$, for~$i = 1$, \dots,~$N$, are
nonnegative. So we see that each~$K_i$ is a continuous and positive
kernel. Mercer's theorem (see for instance Courant and Hilbert
\cite{CourantHilbert1953}) then implies that~$K_i$ is trace-class, its
trace being the sum of all its eigenvalues. So for each~$i = 1$,
\dots,~$N$, the series
\begin{equation}
\label{eq:kernel-series}
\sum_{u \in L^*} \hat{f}_{ii}(u)
\end{equation}
converges, and since each summand is nonnegative, it converges
absolutely.

Suppose now that~$i$, $j = 1$, \dots,~$N$ are so that~$i \neq
j$. Since the matrices of Fourier transforms are nonnegative, for
all~$u \in \R^n$ we have that the matrix
\[
\begin{pmatrix}
\hat{f}_{ii}(u)&\hat{f}_{ij}(u)\\[4pt]
\overline{\hat{f}_{ij}(u)}&\hat{f}_{jj}(u)
\end{pmatrix}
\]
is positive semidefinite, and this in turn implies
that~$|\hat{f}_{ij}(u)|^2 \leq \hat{f}_{ii}(u) \hat{f}_{jj}(u)$ for
all~$u \in \R^n$. Using then the convergence of the
series~\eqref{eq:kernel-series} and the Cauchy-Schwarz inequality, one
gets
\[
\sum_{u \in L^*} |\hat{f}_{ij}(u)| \leq \sum_{u \in L^*}
(\hat{f}_{ii}(u) \hat{f}_{jj}(u))^{1/2} \leq \biggl(\sum_{u \in L^*}
\hat{f}_{ii}(u)\biggr)^{1/2} \biggl(\sum_{u \in L^*}
\hat{f}_{jj}(u)\biggr)^{1/2},
\]
and we see that in fact for all~$i$, $j = 1$, \dots,~$N$ the series
\[
\sum_{u \in L^*} \hat{f}_{ij}(u)
\]
converges absolutely.

This convergence result shows that the sum in~\eqref{eq:gij-l2}
converges absolutely and uniformly for all~$x \in \R^n / L$. This
means that the function defined by this sum is a continuous function,
and since~$g_{ij}$ is also a continuous function, and
in~\eqref{eq:gij-l2} we have convergence in the~$L^2$ sense, we must
also then have pointwise convergence, as we aimed to establish.
\end{proof}
 
With this we may give a proof of Theorem~\ref{thm:space-general}:

\begin{proof}[Proof of Theorem~\ref{thm:space-general}]
Using Lemma~\ref{lem:poisson}, we may repeat the proof of
Theorem~\ref{thm:space} given before, proving the theorem for
continuous functions of bounded support. To extend the proof also to
continuous~$L^1$ functions we use the following trick.

Let~$f\colon \R^n \to \R^{N \times N}$ be a continuous and~$L^1$
function satisfying the hypothesis of the theorem. For each~$T > 0$
consider the function~$g^T\colon \R^n \to \R^{N \times N}$ defined
such that
\[
g^T(x) = \frac{\vol(B(0, T) \cap B(x, T))}{\vol B(0, T)} f(x),
\]
where~$B(p, T)$ is the ball of radius~$T$ centered at~$p$.

It is easy to see that~$g^T$ is a continuous function of bounded
support. It is also clear that it satisfies condition~(iii) from
the statement of the theorem. We now show that~$g^T$ is a function of
positive type, that is, it satisfies condition~(ii).

For this pick any points~$x_1$, \dots,~$x_m \in
\R^n$. Let~$\chi_i\colon \R^n \to \{0,1\}$ be the characteristic
function of~$B(x_i, T)$ and denote by~$\langle \phi, \psi \rangle$ the
standard inner product between functions~$\phi$ and~$\psi$ in the
Hilbert space~$L^2(\R^n)$. Then
\[
\begin{split}
g^T(x_i - x_j)&= \frac{\vol(B(0, T) \cap B(x_i - x_j, T))}{\vol B(0,
  T)} f(x_i - x_j)\\
&= \frac{\vol(B(x_i, T) \cap B(x_j, T))}{\vol B(0, T)} f(x_i - x_j)\\
&= \frac{\langle \chi_i, \chi_j \rangle}{\vol B(0, T)} f(x_i - x_j).
\end{split}
\]

This shows that the matrix~$\bigl(g^T(x_i - x_j)\bigr)_{i,j=1}^m$ is
positive semidefinite, being the Hadamard product, i.e.\ entrywise
product, of two positive semidefinite matrices. We therefore have
that~$g^T$ is of positive type.

Now,~$g^T$ is a continuous function of positive type and bounded
support, satisfying condition~(iii). It is very possible, however,
that~$g^T$ does not satisfy condition~(i), and so the conclusion of
the theorem may not apply to~$g^T$. Let us now fix this
problem. 

Notice that~$g^T_{ij}$ converges pointwise to~$f_{ij}$ as~$T \to
\infty$. Moreover, for all~$T > 0$ we have~$|g^T_{ij}(x)| \leq
|f_{ij}(x)|$. It then follows from Lebesgue's dominated convergence
theorem that~$\hat{g}^T_{ij}(0) \to \hat{f}_{ij}(0)$ as~$T \to
\infty$. This means that there exists a number~$T_0 > 0$ such that 
for each~$T \geq T_0$ we may pick a number~$\alpha(T) \geq 1$ so
that the function~$h^T\colon \R^n \to \C^{N \times N}$ such that
\[
\begin{array}{ll}
h^T_{ii}(x) = \alpha(T) g^T_{ii}(x)&\text{for~$i = 1$, \dots,~$N$},\\[2pt]
h^T_{ij}(x) = g^T_{ij}(x)&\text{for~$i$, $j = 1$, \dots,~$N$ with~$i
  \neq j$}
\end{array}
\]
for all~$x \in \R^n$ satisfies condition~(ii). We may moreover pick
the numbers~$\alpha(T)$ in such a way that~$\lim_{T \to \infty}
\alpha(T) = 1$.

It is also easy to see that each function~$h^T$ is of positive type
and bounded support and satisfies condition~(iii). Hence the
conclusion of the theorem applies for each~$h^T$, and so for every~$T
\geq T_0$ we see that
\[
M_T = \max\{\, h^T_{ii}(0) : \text{$i = 1$, \dots,~$N$}\,\}
\]
is an upper bound for the density of any packing of translated copies
of~$\Kcal_1$, \dots,~$\Kcal_N$. But then, since~$g^T_{ii}(0) =
f_{ii}(0)$ for all~$T \geq 0$, and since~$\lim_{T \to \infty}
\alpha(T) = 1$, we see that
\[
\max\{\, f_{ii}(0) : \text{$i = 1$, \dots,~$N$}\, \}
= \lim_{T \to \infty} M_T,
\]
finishing the proof.
\end{proof}

\section{Computations for binary spherical cap packings}
\label{sec:unit sphere computations}

In this and the next section we describe how we obtained the
numerical results of Sections~\ref{ssec:computations sphere}
and~\ref{ssec:computations space}. Our approach is computational: to
apply Theorems~\ref{thm:sphere} and~\ref{thm:space} we use techniques
from semidefinite programming and polynomial optimization.

We start by briefly discussing the case of binary spherical cap
packings. Next we will discuss the more computationally challenging
case of binary sphere packings.

It is a classical result of Luk\'acs (see e.g.\ Theorem~1.21.1 in
Szeg\"o~\cite{Szego1938}) that a real univariate polynomial~$p$ of
degree~$2d$ is nonnegative on the interval~$[a, b]$ if and only if
there are real polynomials~$q$ and~$r$ such that~$p(x) = (q(x))^2 +
(x-a) (b-x) (r(x))^2$. This characterization is useful when we combine
it with the elementary but powerful observation (discovered
independently by several authors, cf.~Laurent~\cite{Laurent2009}) that
a real univariate polynomial~$p$ of degree~$2d$ is a sum of
squares of polynomials if and only if~$p(x) = v(x)^{\sf T} Q v(x)$ for
some positive semidefinite matrix~$Q$, where~$v(x) = (1, x, \ldots,
x^d)$ is a vector whose components are the monomial basis.

Let~$\alpha_1$, \dots,~$\alpha_N \in (0, \pi]$ be angles and~$d$ be an
integer. Write~$v_0(x) = (1, x, \ldots, x^d)$ and~$v_1(x) = (1, x,
\ldots, x^{d - 1})$.  Using this characterization together with
Theorem~\ref{thm:sphere}, we see that the optimal value of the
following optimization problem gives an upper bound for the density of
a packing of spherical caps with angles~$\alpha_1$,~\dots,~$\alpha_N$.
\medskip

\noindent
{\bf Problem~\spheresdp/.}\enspace 
For~$k = 0$, \dots,~$2d$, find positive semidefinite
matrices~$\bigl(f_{ij,k}\bigr)_{i,j=1}^N$, and for~$i$, $j = 1$,
\dots,~$N$, find~$(d + 1) \times (d+1)$ positive semidefinite
matrices~$Q_{ij}$ and~$d \times d$ positive semidefinite
matrices~$R_{ij}$ that minimize
\[
\max\biggl\{\, \sum_{k=0}^{2d} f_{ii, k} : \text{$i = 1$,
  \dots,~$N$}\,\biggr\}
\]
and are such that
\[
\bigl(f_{ij,0} - w(\alpha_i)^{1/2} w(\alpha_j)^{1/2}\bigr)_{i,j=1}^N
\]
is positive semidefinite and the polynomial identities
\begin{equation}
\label{eq:sphere-poly}
\vcenter{\displaywidth=\hsize\advance\displaywidth by-1cm%
\displaylines{\noalign{\vskip5pt}
\qquad\sum_{k=0}^{2d} f_{ij,k} P_k^n(u) + \langle Q_{ij}, v_0(u) v_0(u)^\tsp
\rangle\hfill\cr
\hfill{} + \langle R_{ij}, (u + 1)(\cos(\alpha_i + \alpha_j) - u)
v_1(u) v_1(u)^\tsp \rangle = 0\qquad\cr
}}
\end{equation}
are satisfied for~$i$, $j = 1$, \dots,~$N$.\hfill$\vartriangleleft$
\medskip

Above,~$\langle A, B\rangle$ denotes the trace inner product between
matrices~$A$ and~$B$. Problem~\spheresdp/ is a semidefinite programming
problem, as the polynomial identities~\eqref{eq:sphere-poly} can 
each be expressed as~$2d + 1$ linear constraints on the entries of the
matrices involved. Indeed, to check that a polynomial is identically
zero, it suffices to check that the coefficient of each monomial~$1$,
$x$, \dots,~$x^{2d}$ is zero, and for each such monomial we get a
linear constraint.

In the above, we work with the standard monomial basis~$1$, $x$,
\dots,~$x^{2d}$, but we could use any other basis of the space of
polynomials of degree at most~$2d$, both to define the vectors~$v_0$
and~$v_1$ and to check the polynomial
identity~\eqref{eq:sphere-poly}. Such a change of basis does not
change the problem from a formal point of view, but can drastically
improve the performance of the solvers used. In our computations for
binary spherical cap packings it was enough to use the standard
monomial basis. We will see in the next section, when we present our
computations for the Euclidean space, that a different choice of basis
is essential.

We reported in Section~\ref{ssec:computations sphere} on our
calculations for~$N = 1$, and~$2$ and~$n = 3$, $4$, and~$5$. The bounds, for the
angles under consideration, do not seem to improve beyond~$d = 25$, so
we use this value for~$d$ in all computations. To obtain these
bounds we used the solver SDPA-QD, which works with quadruple precision
floating point numbers, from the SDPA family~\cite{SDPA}.

\section{Computations for binary sphere packings}
\label{sec:Euclidean space computations}

In this section we discuss our computational approach to find upper
bounds for the density of binary sphere packings using
Theorem~\ref{thm:space}. This is a more difficult application of
semidefinite programming and polynomial optimization techniques than
the one described in Section~\ref{sec:unit sphere computations}.

It is often the case in applications of sum of squares
techniques that, if one formulates the problems carelessly, high
numerical instability invalidates the final results, or even numerical
results cannot easily be obtained.  This raises questions of how to
improve the formulations used and the precision of the computations,
so that we may provide \emph{rigorous} bounds. We also address these
questions and, since the techniques we use and develop might be
of interest to the reader who wants to perform computations in
polynomial optimization,  we include some details.

\subsection{Theorem~\ref{thm:space} for multiple-size sphere packings}

In the case of multiple-size sphere packings, Theorem~\ref{thm:space}
can be simplified. The key observation here is that, when all the
bodies~$\Kcal_i$ are spheres, then condition~(iii) depends only on the
norm of the vector~$x$. More specifically, if each~$\Kcal_i$ is a
sphere of radius~$r_i$, then~$\Kcal_i^\circ \cap (x + \Kcal_j^\circ) =
\emptyset$ if and only if~$\|x\| \geq r_i + r_j$.

So in Theorem~\ref{thm:space} one can choose to restrict oneself to
radial functions. A function~$f\colon \R^n \to \C$ is \defi{radial} if
the value of~$f(x)$ depends only on the norm of~$x$. If~$f\colon
\R^n \to \C$ is radial, for~$t \geq 0$ we denote by~$f(t)$ the common
value of~$f$ for vectors of norm~$t$.

The Fourier transform~$\hat{f}(u)$ of a radial function~$f$ also
depends only on the norm of~$u$; in other words, the Fourier transform
of a radial function is also radial. By restricting ourselves to
radial functions, we obtain the following version of
Theorem~\ref{thm:space}.

\begin{theorem}
\label{thm:space-binary}
Let~$r_1$, \dots,~$r_N > 0$ and let~$f\colon \R^n \to \R^{N \times N}$
be a matrix-valued function whose every component~$f_{ij}$ is a radial
Schwartz function. Suppose~$f$ satisfies the following conditions:
\begin{itemize}
\item[(i)] the matrix~$\bigl(\hat{f}_{ij}(0) - (\vol B(r_i))^{1/2}
  (\vol B(r_j))^{1/2}\bigr)_{i,j=1}^N$ is positive semidefinite,
  where~$B(r)$ is the ball of radius~$r$ centered at the origin;

\item[(ii)] the matrix of Fourier
  transforms~$\bigl(\hat{f}_{ij}(t)\bigr)_{i,j=1}^N$  is positive
  semidefinite for every~$t > 0$;

\item[(iii)] $f_{ij}(w) \leq 0$ if~$w \geq r_i + r_j$, for~$i$, $j =
  1$, \dots,~$N$.
\end{itemize}
Then the density of any packing of spheres of radii~$r_1$,
\dots,~$r_N$ in the Euclidean space~$\R^n$ is at
most~$\max\{\, f_{ii}(0) : \text{$i = 1$, \dots,~$N$}\,\}$.
\end{theorem}

One might ask whether the restriction to radial functions worsens the
bound of Theorem~\ref{thm:space}. For spheres, this is not the case.
Indeed, suppose each body~$\Kcal_i$ is a sphere. If~$f\colon \R^n \to
\R^{N \times N}$ is a function satisfying the conditions of the
theorem, then its radialized version, the function
\[
\overline{f}(x) = \int_{S^{n-1}} f(\|x\| \xi)\, d\omega_n(\xi),
\]
also satisfies the conditions of the theorem, and it provides the same
upper bound. This shows in particular that, for the case of
multiple-size sphere packings, Theorem~\ref{thm:space-binary} is
equivalent to Theorem~\ref{thm:space}.

\subsection{A semidefinite programming formulation}

To simplify notation and because it is the case of our main interest
we now take~$N = 2$. Everything in the following also goes through for
arbitrary~$N$ with obvious modifications.

To find a function~$f$ satisfying the conditions of
Theorem~\ref{thm:space-binary} we specify~$f$ via its Fourier
transform. Let~$d \geq 0$ be an odd integer and consider the even
function~$\varphi\colon \R_{\geq 0} \to \R^{2 \times 2}$ such that
\[
\varphi_{ij}(t) = \sum_{k = 0}^d a_{ij, k} t^{2k},
\]
where each~$a_{ij, k}$ is a real number and~$a_{ij, k} = a_{ji, k}$
for all~$k$. We set the Fourier transform of~$f$ to be
\[
\hat{f}_{ij}(u) = \varphi_{ij}(\|u\|) e^{-\pi \|u\|^2}.
\]
Notice that each~$\hat{f}_{ij}$ is a Schwartz function, so its Fourier
inverse is also Schwartz.

The reason why we choose this form for the Fourier transform of~$f$
is that it makes it simple to compute~$f$ from its Fourier
transform by using the following result.

\begin{lemma}
\label{lem:inverse}
We have that
\begin{equation}
\label{eq:Laquerre}
\int_{\R^n} \|u\|^{2k} e^{-\pi\|u\|^2} e^{2\pi i u \cdot x}\, du =
k!\, \pi^{-k} e^{-\pi \|x\|^2} L_k^{n/2-1}(\pi \|x\|^2),
\end{equation}
where~$L_k^{n/2-1}$ is the Laguerre polynomial of degree~$k$ with
parameter~$n/2 - 1$.
\end{lemma}

For background on Laguerre polynomials, we refer the reader to the
book by Andrews, Askey, and Roy~\cite{AndrewsAskeyRoy1999}.

\begin{proof}
  With $f(u) = \|u\|^{2k} e^{- \pi \|u\|^2}$, the left hand side of
  \eqref{eq:Laquerre} is equal to $\hat{f}(-x)$.  By \cite[Theorem~9.10.3]{AndrewsAskeyRoy1999} we have
\[
\hat{f}(-x) = 2 \pi \|x\|^{1 - n/2} \int_0^\infty s^{2k} e^{- \pi s^2} J_{n / 2 - 1}(2 \pi s \|x\|) s^{n/2} \, ds,
\]
where $J_{n / 2 - 1}$ is the Bessel function of the first kind with
parameter $n / 2 - 1$.  Using \cite[Corollary~4.11.8]{AndrewsAskeyRoy1999} we see that this is equal to
\begin{equation}
\label{eq:cor4118}
\pi^{-k} \frac{\Gamma(k + n/2)}{\Gamma(n / 2)} e^{-\pi \|x\|^2} {}_1F_1\Bigl({-k\atop n/2}; \pi \|x\|^2\Bigr),
\end{equation}
where ${}_1F_1$ is a hypergeometric series.

By \cite[(6.2.2)]{AndrewsAskeyRoy1999} we have
\[
{}_1F_1\Bigl({-k\atop n/2}; \pi \|x\|^2\Bigr) = \frac{k!}{(n/2)_k}
L_k^{n/2-1}(\pi \|x\|^2),
\]
where~$(n/2)_k = (n/2) (1 + n/2) \cdots (k - 1 + n / 2)$.

By substituting this in \eqref{eq:cor4118}, and using the property
that $\Gamma(x+1) = x \Gamma(x)$ for all $x \neq 0, -1, -2, \ldots$,
we obtain the right hand side of \eqref{eq:Laquerre} as desired.
\end{proof}

So we have
\[
f_{ij}(x) = \int_{\R^n} \varphi_{ij}(\|u\|) e^{-\pi\|u\|^2} e^{2\pi i u
  \cdot x}\, du
= \sum_{k=0}^d a_{ij, k} \, k!\, \pi^{-k} e^{-\pi \|x\|^2} L_k^{n/2-1}(\pi \|x\|^2).
\]
Notice that it becomes clear that~$f_{ij}$ is indeed real-valued, as
required by the theorem.

Consider the polynomial
\[
p(t) = \sum_{k=0}^d a_k t^{2k}.
\]
According to Lemma~\ref{lem:inverse}, if~$g(x)$ is the Fourier inverse
of~$\hat{g}(u) = p(\|u\|) e^{-\pi \|u\|^2}$, then~$g(\|x\|) = q(\|x\|)
e^{-\pi \|x\|^2}$, where
\[
q(w) = \sum_{k=0}^d a_k \, k!\, \pi^{-k} L_k^{n/2-1}(\pi w^2)
\]
is a univariate polynomial. We denote the polynomial~$q$ above
by~$\finv[p]$. Notice that~$\finv[p]$ is obtained from~$p$ via a
linear transformation, i.e., its coefficients are linear combinations
of the coefficients of~$p$. With this notation we have 
\[
f_{ij}(x) = \finv[\varphi_{ij}](\|x\|) e^{-\pi \|x\|^2}.
\]

Let
\begin{equation}
\label{eq:sigma}
\sigma(t, y_1, y_2) = \sum_{i,j=1}^2 \sum_{k=0}^d a_{ij, k} t^{2k} y_i
y_j.
\end{equation}
If this polynomial is a sum of squares, then it is nonnegative
everywhere, and hence the
matrices~$\bigl(\varphi_{ij}(t)\bigr)_{i,j=1}^2$ are positive
semidefinite for all~$t \geq 0$. This implies that~$f$ satisfies
condition~(ii) of Theorem~\ref{thm:space-binary}. (The converse is
also true, that if the
matrices~$\bigl(\varphi_{ij}(t)\bigr)_{i,j=1}^2$ are positive
semidefinite for all~$t \geq 0$, then~$\sigma$ is a sum of squares;
For a proof see Choi, Lam, Reznick \cite{ChoiLamReznick1980}. This fact is
related to the Kalman-Yakubovich-Popov lemma in systems and control; see the discussion in
Aylward, Itani, and Parrilo~\cite{AylwardItaniParrilo2007}.)

Moreover, we may recover~$\varphi$, and hence~$\hat{f}$,
from~$\sigma$. Indeed we have
\begin{equation}
\label{eq:phi-from-sigma}
\begin{split}
\varphi_{11}(t) &= \sigma(t, 1, 0),\\
\varphi_{22}(t) &= \sigma(t, 0, 1),\quad\text{and}\\
\varphi_{12}(t) &= (1/2) (\sigma(t, 1, 1) - \sigma(t, 1, 0) -
\sigma(t, 0, 1)).
\end{split}
\end{equation}
So we can express condition~(i) of Theorem~\ref{thm:space-binary} in
terms of~$\sigma$. We may also express condition~(iii) in
terms of~$\sigma$, since it can be translated as
\begin{equation}
\label{eq:non-pos-poly}
\finv[\varphi_{ij}](w) \leq 0\quad\text{for all~$w \geq r_i + r_j$
  and~$i$, $j = 1$, $2$ with~$i \leq j$}.
\end{equation}

If we find a polynomial~$\sigma$ of the form~\eqref{eq:sigma} that
is a sum of squares, is such that
\begin{equation}
\label{eq:sigma-ii}
\bigl(\varphi_{ij}(0) - (\vol B(r_i))^{1/2} (\vol
B(r_j))^{1/2}\bigr)_{i,j=1}^2
\end{equation}
is positive semidefinite, and satisfies~\eqref{eq:non-pos-poly}, then
the density of a packing of spheres of radii~$r_1$ and~$r_2$ is upper
bounded by
\[
\max\{ \finv[\varphi_{11}](0), \finv[\varphi_{22}](0) \}.
\]

We may encode conditions~\eqref{eq:non-pos-poly} in terms of sums of
squares polynomials (cf.~Section~\ref{sec:unit sphere computations}),
and therefore we may encode the problem of finding a~$\sigma$ as above
as a semidefinite programming problem, as we show now.

Let~$P_0$, $P_1$, \dots\ be a sequence of univariate polynomials where
polynomial~$P_k$ has degree~$k$. Consider the vector of
polynomials~$v$, which has entries indexed by~$\{ 0, \ldots,
\floor{d/2}\}$ given by
\[
v(t)_k = P_k(t^2)
\]
for~$k = 0$, \dots,~$\floor{d/2}$. We also write~$V(t) = v(t)
v(t)^\tsp$.

Consider also the vector of polynomials~$m$ with entries indexed
by~$\{ 1, 2 \} \times \{ 0, \ldots, \floor{d/2}\}$ given by
\[
m(t, y_1, y_2)_{i, k} = P_k(t^2) y_i
\]
for~$i$, $j = 1$, $2$ and~$k = 0$, \dots,~$\floor{d/2}$.

Since~$\sigma$ is an even polynomial, it is a sum of squares if and
only if there are positive semidefinite matrices~$S_0$, $S_1 \in
\R^{(d+1) \times (d+1)}$ such that
\[
\sigma(t, y_1, y_2) = \langle S_0, m(t, y_1, y_2) m(t, y_1, y_2)^\tsp
\rangle + \langle S_1, t^2 m(t, y_1, y_2) m(t, y_1, y_2)^\tsp \rangle.
\]

From the matrices~$S_0$ and~$S_1$ we may then recover~$\varphi_{ij}$
and also~$\finv[\varphi_{ij}]$. A more convenient way for
expressing~$\varphi_{ij}$ in terms of~$S_0$ and~$S_1$ is as
follows. Consider the matrices
\[
Y_{11} = \begin{pmatrix}
1&0\\
0&0
\end{pmatrix},
\quad
Y_{22} = \begin{pmatrix}
0&0\\
0&1
\end{pmatrix},
\quad\text{and}\quad
Y_{12} = \begin{pmatrix}
0&1/2\\
1/2&0
\end{pmatrix}.
\]
Then
\[
\varphi_{ij}(t) = \langle S_0, V(t) \otimes Y_{ij} \rangle
+ \langle S_1, t^2 V(t) \otimes Y_{ij} \rangle
\]
and
\[
\finv[\varphi_{ij}](w) = \langle S_0, \finv[V(t)](w) \otimes Y_{ij} \rangle
+ \langle S_1, \finv[t^2 V(t)](w) \otimes Y_{ij} \rangle,
\]
where~$\finv$, when applied to a matrix, is applied to each entry
individually.

With this, we may consider the following semidefinite programming
problem for finding a polynomial~$\sigma$ satisfying the conditions we
need.
\medskip

\noindent
{\bf Problem~\spacesdp/.}\enspace
Find $(d+1) \times (d+1)$ real positive semidefinite
matrices~$S_0$ and~$S_1$, and $(\floor{d/2} + 1) \times
(\floor{d/2}+1)$ real positive semidefinite matrices~$Q_{11}$, $Q_{22}$,
and~$Q_{12}$ that minimize
\[
\begin{split}
\max\{&\langle S_0, \finv[V(t)](0) \otimes Y_{11}\rangle + \langle
S_1, \finv[t^2 V(t)](0) \otimes Y_{11}\rangle,\\
&\langle S_0, \finv[V(t)](0) \otimes Y_{22}\rangle + \langle
S_1, \finv[t^2 V(t)](0) \otimes Y_{22}\rangle\}
\end{split}
\]
and are such that
\begin{equation}
\label{eq:mat-psd}
\bigl(\langle S_0, V(0) \otimes Y_{ij}\rangle - (\vol B(r_i))^{1/2}
(\vol(B(r_j))^{1/2}\bigr)_{i,j=1}^2
\end{equation}
is positive definite and the polynomial identities
\begin{equation}
\label{eq:non-pos-sdp}
\vcenter{\displaywidth=\hsize\advance\displaywidth by-1cm%
\displaylines{\noalign{\vskip5pt}
\qquad\langle S_0, \finv[V(t)](w) \otimes Y_{ij} \rangle
+ \langle S_1, \finv[t^2 V(t)](w) \otimes Y_{ij} \rangle
\hfill\cr
\hfill{} + \langle Q_{ij}, (w^2 - (r_i + r_j)^2) V(w)\rangle = 0\phantom{,}\qquad\cr
}}
\end{equation}
are satisfied for~$i$, $j = 1$, $2$ and~$i \leq j$.\hfill$\vartriangleleft$
\medskip

Any solution to this problem gives us a polynomial~$\sigma$ of the
shape~\eqref{eq:sigma} which is a sum of squares and satisfies
conditions~\eqref{eq:non-pos-poly} and~\eqref{eq:sigma-ii}, and so the
optimal value is an upper bound for the density of any packing of
spheres of radius~$r_1$ and~$r_2$. There might be, however,
polynomials~$\sigma$ satisfying these conditions that cannot be
obtained as feasible solutions to Problem~\spacesdp/, since
condition~\eqref{eq:non-pos-sdp} is potentially more restrictive than
condition~\eqref{eq:non-pos-poly} (compare Problem~\spacesdp/ above
with Luk\'acs' result mentioned in Section~\ref{sec:unit sphere
  computations}). In our practical computations this restriction was
not problematic and we found very good functions.

Observe also that Problem~\spacesdp/ is really a semidefinite
programming problem. Indeed, the polynomial identities
in~\eqref{eq:non-pos-sdp} can each be represented as~$d + 1$ linear
constraints in the entries of the matrices~$S_i$ and~$Q_{ij}$. This is
the case because testing whether a polynomial is identically zero is
the same as testing whether each monomial has a zero coefficient and
so, since all our polynomials are even and of degree~$2d$, we need
only check if the coefficients of the monomials~$x^{2k}$ are zero
for~$k = 0$, \dots,~$d$.

\subsection{Numerical results}

When solving Problem~\spacesdp/, we need to choose a
sequence $P_0$, $P_1$, \dots\ of polynomials. A choice which works
well in practice is
\[
P_k(t) = \mu_k^{-1} L_k^{n/2-1}(2 \pi t),
\]
where~$\mu_k$ is the absolute value of the coefficient
of~$L_k^{n/2-1}(2 \pi t)$ with largest absolute value. We observed in
practice that the standard monomial basis performs poorly.

To represent the polynomial identities in~\eqref{eq:non-pos-sdp} as
linear constraints we may check that each monomial~$x^{2k}$ of the
resulting polynomial has coefficient zero. We may use, however, any
basis of the space of even polynomials of degree at most~$2d$ to
represent such identities. Given such a basis, we expand each
polynomial in it and check that the expansion has only zero
coefficients. The basis we use to represent the identities
is~$P_0(t^2)$, $P_1(t^2)$, \dots,~$P_d(t^2)$, which we observed to
work much better than~$t^0$, $t^2$, \dots,~$t^{2d}$.  Notice that no
extra variables are necessary if we use a different basis to represent
the identities. We need only keep, for each polynomial in the
matrices~$\finv[V(t)](w)$, $\finv[t^2 V(t)](w)$, $w^2 V(w)$,
and~$V(w)$, its expansion in the basis we want to use.

The plot in Figure~\ref{fig:space-plots} was generated by solving
Problem~\spacesdp/ with~$d = 31$ using the solver
SDPA-GMP from the SDPA family~\cite{SDPA}. The input for the solver was generated by a
SAGE~\cite{SAGE} program working with floating-point arithmetic and
precision of~$256$ bits. For each dimension~$2$, \dots,~$5$ we solved
Problem~\spacesdp/ with~$r_1 = r / 1000$ and~$r_2 = 1$ for~$r = 200$,
\dots,~$1000$; the reason we start with~$r = 200$ is that for
smaller values of~$r$ the solver runs into numerical stability
problems. We also note that the solver has failed to solve some of the
problems, and these points have been disconsidered when generating the
plot. The number of problems that could not be solved was small
though: for~$n = 2$ all problems could be solved, for~$n = 3$ there
were~$6$ failures, for~$n = 4$ we had~$18$ failures, and finally
for~$n = 5$ the solver failed for~$137$ problems.

With our methods we can achieve higher values for~$d$, but we noticed
that the bound does not improve much after~$d = 31$. For instance, in
dimension~$2$ for~$r_1 = 1/2$ and~$r_2 = 1$, we obtain the
bound~$0.9174466\ldots$ for~$d = 31$ and the bound~$0.9174426\ldots$
for~$d = 51$.\medskip

\centerline{$*$\quad$*$\quad$*$}
\medskip

\noindent
In the previous account of how the plot in
Figure~\ref{fig:space-plots} was generated, we swept under the rug
all precision issues. We generate the data for the solver using
floating-point arithmetic, and the solver also uses floating-point
arithmetic. We cannot therefore be sure that the optimal value found
by the solver gives a valid bound at all.

If we knew \textit{a priori} that Problem~\spacesdp/ is strictly
feasible (that is, that it admits a solution in which the
matrices~$S_i$ and~$Q_{ij}$ are positive definite), and if we had some
control over the dual solutions, then we could use semidefinite
programming duality to argue that the bounds we compute are rigorous;
see for instance Gijswijt \cite[Chapter 7.2]{Gijswijt2005} for an
application of this approach in coding theory. The matter is however
that we do not know that Problem~\spacesdp/ is strictly feasible,
neither do we have knowledge about the dual solutions. In fact, most
of our approach to provide rigorous bounds consists in finding a
strictly feasible solution.

A naive idea to turn the bound returned by the solver into a rigorous
bound would be to simply project a solution returned by the solver
onto the subspace given by the constraints
in~\eqref{eq:non-pos-sdp}. If the original solution is of good
quality, then this would yield a feasible solution.

There are two
problems with this approach, though. The first problem is that the
matrices returned by the solver will have eigenvalues too close to
zero, and therefore after the projection they might not be positive
semidefinite anymore. We discuss how to handle this issue below.

The second problem is that to obtain a rigorous bound one would need
to perform the projection using symbolic computations and rational
arithmetic, and the computational cost is just too big. For instance,
we failed to do so even for~$d = 7$.

Our approach avoids projecting the solution using symbolic
computations. Here is an outline of our method.
\begin{enumerate}
\item Obtain a solution to the problem with objective value close
  the optimal value returned by the solver, but in which every
  matrix~$S_i$ and~$Q_{ij}$ is positive definite by a good margin and the
  maximum violation of the constraints is very small.

\item Approximate matrices~$S_i$ and~$Q_{ij}$ by rational positive
  semidefinite matrices~$\bar{S}_i$ and~$\bar{Q}_{ij}$ having
  minimum eigenvalues at least~$\lambda_i$ and~$\mu_{ij}$, respectively.

\item Compute a bound on how much
  constraints~\eqref{eq:non-pos-sdp} are violated by~$\bar{S}_i$
  and~$\bar{Q}_{ij}$ using rational arithmetic.  If the maximum violation
  of the constraints is small compared to the bounds~$\lambda_i$
  and~$\mu_{ij}$ on the minimum eigenvalues, then we may be sure that the
  solution can be changed into a feasible solution without changing
  its objective value too much.
\end{enumerate}
We now explain how each step above can be accomplished.

First, most likely the matrices~$S_i$, $Q_{ij}$ returned by the solver
will have eigenvalues very close to zero, or even slightly negative
due to the numerical method which might allow infeasible steps.

To obtain a solution with positive definite matrices we may use the
following trick (cf.~L\"ofberg~\cite{Lofberg2011}). We solve
Problem~\spacesdp/ to find its optimal value, say~$z^*$. Then we solve
a feasibility version of Problem~\spacesdp/ in which the objective
function is absent, but we add a constraint to ensure that
\[
\begin{split}
\max\{&\langle S_0, \finv[V(t)](0) \otimes Y_{11}\rangle + \langle
S_1, \finv[t^2 V(t)](0) \otimes Y_{11}\rangle,\\
&\langle S_0, \finv[V(t)](0) \otimes Y_{22}\rangle + \langle
S_1, \finv[t^2 V(t)](0) \otimes Y_{22}\rangle\} \leq z^* + \eta,
\end{split}
\]
where~$\eta > 0$ should be small enough so that we do not jeopardize
the objective value of the solution, but not too small so that a good
strictly feasible solution exists. (We take~$\eta = 10^{-5}$, which
works well for the purpose of making a plot.) The trick here is that
most semidefinite programming solvers, when solving a feasibility
problem, will return a strictly feasible solution --- the analytical
center ---, if one can be found.

This partially addresses step~(1), because though the solution we find
will be strictly feasible, it might violate the constraints too
much. To quickly obtain a solution that violates the constraints only
slightly, we may project our original solution onto the subspace given
by constraints~\eqref{eq:non-pos-sdp} using floating-point
arithmetic of high enough precision. If the solution returned by the
solver had good precision to begin with, then the projected solution
will still be strictly feasible.

As an example, for our problems with~$d = 31$, SDPA-GMP returns
solutions that violate the constraints by at most~$10^{-30}$. By doing
a projection using floating-point arithmetic with~$256$ bits of
precision in SAGE, we can bring the violation down to
about~$10^{-70}$ without affecting much the eigenvalues of the matrices.

So we have addressed step~(1). For step~(2) we observe that simply
converting the floating-point matrices~$S_i$, $Q_{ij}$ to rational
matrices would work, but then we would be in trouble to estimate the
minimum eigenvalues of the resulting rational matrices in a rigorous
way. Another idea of how to make the conversion is as follows.

Say we want to approximate floating-point matrix~$A$ by a rational
matrix~$\bar{A}$. We start by computing numerically an
approximation to the least eigenvalue of~$A$. Say~$\tilde{\lambda}$ is
this approximation. We then use binary search in the
interval~$[\tilde{\lambda} / 2, \tilde{\lambda}]$ to find the
largest~$\lambda$ so that the matrix~$A - \lambda I$ has a Cholesky
decomposition; this we do using floating-point arithmetic of high
enough precision. If we have this largest~$\lambda$, then
\[
A = L L^\tsp + \lambda I
\]
where~$L$ is the Cholesky factor of~$A - \lambda I$. Then we
approximate~$L$ by a rational matrix~$\bar{L}$ and we set
\[
\bar{A} = \bar{L} \bar{L}^\tsp + \lambda I,
\]
obtaining thus a rational approximation of~$A$ and a bound on its
minimum eigenvalue.

Our idea for step~(3) is to compare the maximum violation of
constraints~\eqref{eq:non-pos-sdp} with the minimum eigenvalues of the
matrices. To formalize this idea, suppose that
constraints~\eqref{eq:non-pos-sdp} are slightly violated
by~$\bar{S}_i$, $\bar{Q}_{ij}$. So for instance we have
\begin{equation}
\label{eq:first-const}
\vcenter{\displaywidth=\hsize\advance\displaywidth by-1cm%
\displaylines{\noalign{\vskip5pt}
\qquad\langle \bar{S}_0, \finv[V(t)](w) \otimes Y_{11} \rangle
+ \langle \bar{S}_1, \finv[t^2 V(t)](w) \otimes Y_{11} \rangle
\hfill\cr
\hfill{} + \langle \bar{Q}_{11}, (w^2 - (2r_1)^2) V(w)\rangle = p,\qquad\cr\noalign{\vskip5pt}
}}
\end{equation}
where~$p$ is an even polynomial of degree at most~$2d$. Notice that we
may compute an upper bound on the absolute values of the coefficients
of~$p$ using rational arithmetic.

To fix this constraint we may distribute the coefficients of~$p$ in
the matrices~$\bar{S}_0$ and~$\bar{Q}_{11}$ (a very similar idea was
presented by L\"ofberg~\cite{Lofberg2009}). To make things precise,
for~$k = 1$, \dots,~$d$ write
\[
\begin{split}
i(k) &= \min\{ \floor{d/2}, k - 1 \},\\
j(k) &= k - 1 - i(k).
\end{split}
\]
Pairs~$(i(k), j(k))$ correspond to entries of the
matrix~$V(w)$. Notice that the polynomial~$(w^2 - (2r_1)^2) V(w)_{i(k)
  j(k)}$ has degree~$2k$.

So the polynomials
\[
\begin{split}
R_0 &= \finv[V(t)_{00}](w),\\
R_1 &= (w^2 - (2r_1)^2) V(w)_{i(1) j(1)},\\
&\setbox0=\hbox{${}={}$}\hbox to\wd0{\hfil\vdots\hfil}\\
R_d &= (w^2 - (2r_1)^2) V(w)_{i(d) j(d)}
\end{split}
\]
form a basis of the space of even polynomials of degree at
most~$2d$. We may then express our polynomial~$p$ in this basis as
\[
p = \alpha_0 R_0 + \cdots + \alpha_d R_d.
\]
Now, we subtract~$\alpha_0$ from~$(\bar{S}_0)_{(1, 0), (1, 0)}$
and~$\alpha_k$ from~$(\bar{Q}_{11})_{i(k)j(k)}$, for~$k = 1$,
\dots,~$d$. The resulting matrices satisfy
constraint~\eqref{eq:first-const}, and as long as the~$\alpha_k$ are
small enough, they should remain positive semidefinite. More
precisely, it suffices to require that~$d\, \|(\alpha_1, \ldots,
\alpha_d)\|_\infty \leq \mu_{11}$ and~$|\alpha_0| \leq \lambda_0$.

There are two issues to note in our approach. The first one is that it
has to be applied again twice to fix the other two constraints
in~\eqref{eq:non-pos-sdp}. The applications do not conflict with each
other: in each one we change a different matrix~$\bar{Q}_{ij}$ and
different entries of~$\bar{S}_0$. We have to be careful though that we
consider the changes to~$\bar{S}_0$ at once in order to check
that it remains positive semidefinite.

The second issue is how to compute the
coefficients~$\alpha_k$. Computing them explicitly using symbolic
computation is infeasible. One way to do it then is to consider the
basis change matrix between the bases~$x^{2k}$, for~$k = 0$,
\dots,~$d$, and~$R_0$, \dots,~$R_d$, which we denote by~$U$. Then we
know that
\[
\|(\alpha_0, \ldots, \alpha_d)\|_\infty \leq \|U^{-1}\|_\infty \|p\|_\infty,
\]
where~$\|p\|_\infty$ is the~$\infty$-norm of the vector of
coefficients of~$p$ in the basis~$x^{2k}$.

So if we have an upper bound for~$\|U^{-1}\|_\infty$ we are done. To
quickly find such an upper bound, we use an algorithm of
Higham~\cite{Higham1983} (cf.~also Higham~\cite{Higham1987}) which
works for triangular matrices, like~$U$. This bound proved to be good
enough for our purposes.

\section{Improving sphere packing bounds}
\label{sec:improve-ce-detail}

We now prove Theorem~\ref{thm:improve-ce} and show how to use it in
order to compute the bounds presented in Table~\ref{tab:improve-ce}.

\begin{proof}[Proof of Theorem~\ref{thm:improve-ce}]
Let~$x_1$, \dots,~$x_N \in \R^n$ and~$L \subseteq \R^n$ be a lattice
such that
\[
\bigcup_{v \in L} \bigcup_{i=1}^N v + x_i + B
\]
is a sphere packing, where~$B$ is the ball of radius~$1/2$ centered at
the origin. We may assume that, if~$i \neq j$ and~$v \neq 0$, then the
distance between the centers of~$v + x_i + B$ and~$x_j + B$ is greater
than~$1 + \varepsilon_m$. Indeed, we could discard all~$x_i$ that lie
at distance less than~$1 + \varepsilon_m$ from the boundary of the
fundamental parallelotope of~$L$. If the fundamental parallelotope is
big enough (and if it is not, we may consider a dilated version of~$L$
instead), this will only slightly alter the density of the packing,
and the resulting packing will have the desired property.

Consider the sum
\begin{equation}
\label{eq:ce-sum}
\sum_{i,j=1}^N \sum_{v \in L} f(v + x_i - x_j).
\end{equation}

Using the Poisson summation formula, we may rewrite it as
\[
\frac{1}{\vol(\R^n / L)} \sum_{i,j=1}^N \sum_{u \in L^*} \hat{f}(u)
e^{2\pi i u \cdot (x_i - x_j)}.
\]

By discarding all summands in the inner sum above except the one
for~$u = 0$, we see that~\eqref{eq:ce-sum} is at least
\[
\frac{N^2 \vol B}{\vol(\R^n / L)}.
\]

For~$k = 1$, \dots,~$m$, write~$F_k = \{\, (i, j) : \|x_i - x_j\| \in
[1 + \varepsilon_{k-1}, 1 + \varepsilon_k)\,\}$. Then we see
that~\eqref{eq:ce-sum} is at most
\[
N f(0) + \eta_1 |F_1| + \cdots + \eta_m |F_m|.
\]

So we see that
\[
\frac{N \vol B}{\vol(\R^n / L)} \leq f(0) + \eta_1 \frac{|F_1|}{N} +
\cdots + \eta_m \frac{|F_m|}{N}.
\]
Notice that the left-hand side above is exactly the density of our
packing. Now, from the definition of~$M(\varepsilon)$, it is clear
that for~$k = 1$, \dots,~$m$ we have
\[
\frac{|F_1|}{N} + \cdots + \frac{|F_k|}{N} \leq M(\varepsilon_k),
\]
and the theorem follows.
\end{proof}

To find good functions~$f$ satisfying the conditions required by
Theorem~\ref{thm:improve-ce} we used the same approach from
Section~\ref{sec:Euclidean space computations}. We fix an odd positive
integer~$d$ and specify~$f$ via its Fourier transform, writing
\[
\varphi(t) = \sum_{k=0}^d a_k t^{2k}
\]
and setting
\[
\hat{f}(u) = \varphi(\|u\|) e^{-\pi \|u\|^2}.
\]

Using Lemma~\ref{lem:inverse} we then have that
\[
f(x) = \finv[\varphi](\|x\|) e^{-\pi \|x\|^2},
\]
where
\[
\finv[\varphi](w) = \sum_{k=0}^d a_k k!\, \pi^{-k} L_k^{n/2-1}(\pi
w^2)
\]
is a polynomial obtained as a linear transformation of~$\varphi$.

Constraint~(ii), requiring that~$\hat{f}(u) \geq 0$ for all~$u \in
\R^n$, can be equivalently expressed as requiring that the
polynomial~$\varphi$ should be a sum of squares. 

Recalling the result of Luk\'acs mentioned in 
Section~\ref{sec:unit sphere computations},
one may also express constraint~(iii) in terms of sums of squares: one
simply has to require that there exist polynomials~$p_0(w)$
and~$q_0(w)$ such that
\[
\finv[\varphi](w) = -(p_0(w))^2 - (w^2 - (1 + \varepsilon_m)^2) (q_0(w))^2.
\]

In a similar way, one may express constraints~(iv). For instance, for
a given~$k$, we require that there should exist
polynomials~$p_k(w)$ and~$q_k(w)$ such that
\[
\finv[\varphi](w)e^{-\pi(1 + \varepsilon_{k-1})^2} - \eta_1 = -(p_k(w))^2 - (w - (1 +
\varepsilon_{k-1}))((1 + \varepsilon_k) - w)(q_k(w))^2,
\]
and this implies~(iv).

So we may represent the constraints on~$f$ in terms of sums of
squares, and therefore also in terms of semidefinite programming, as
we did in Sections~\ref{sec:unit sphere computations} 
and~\ref{sec:Euclidean space computations}. There is only
the issue that now we want to find a function~$f$ that satisfies
constraints~(i)--(iv) of the theorem and that minimizes the maximum
in~\eqref{eq:ce-lp}. This does not look like a linear objective
function, but since by linear programming duality this maximum is
equal to
\[
\begin{array}{rll}
\min&f(0) + y_1 U(\varepsilon_1) + \cdots + y_m U(\varepsilon_m)\\
&y_i + \cdots + y_m \geq \eta_i&\text{for~$i = 1$, \dots,~$m$},\\
&y_k \geq 0&\text{for~$k = 1$, \dots,~$m$},
\end{array}
\]
we may transform our original problem into a single minimization
semidefinite programming problem, the optimal value of which provides
an upper bound for the densities of sphere packings.

It is still a question how to compute upper bounds
for~$M(\varepsilon)$. For this we use upper bounds on the sizes of
spherical codes. A \defi{spherical code} with \defi{minimum angular
  distance}~$0 < \theta \leq \pi$ is a set~$C \subseteq S^{n-1}$ such
that the angle between any two distinct points in~$C$ is at
least~$\theta$. In other words, a spherical code with minimum angular
distance $\theta$ gives as packing of spherical caps with angle
$\theta/2$. We denote by~$A(n, \theta)$ the maximum cardinality of any
spherical code in~$S^{n-1}$ with minimum angular distance~$\theta$.

There is a simple relation between~$A(n, \theta)$
and~$M(\varepsilon)$. Namely, if~$\varepsilon < \sqrt{2} - 1$, then
\[
M(\varepsilon) \leq A(n, \arccos t(\varepsilon)),
\]
where
\[
t(\varepsilon) = 1 - \frac{1}{2 (1 + \varepsilon)^2}.
\]

To see this, suppose~$x$, $y \in \R^n$ are such that~$\|x\|$, $\|y\|
\in [1, 1 + \varepsilon)$ and~$\|x-y\| \geq 1$. Then by the law of
cosines, if~$\theta$ is the angle between~$x$ and~$y$, we have
\[
\cos \theta = \frac{\|x\|^2 + \|y\|^2 - \|x-y\|^2}{2\|x\|\|y\|}.
\]
The maximum of the right-hand side above for vectors~$x$ and~$y$ such
that~$\|x\|$, $\|y\| \in [1, 1 + \varepsilon)$ and~$\|x-y\| \geq 1$
gives~$t(\varepsilon)$.

Indeed, to maximize the right-hand side above, we may assume
that~$\|x-y\| = 1$. Then
\[
\label{eq:angle}
\cos \theta = \frac{\alpha^2 + \beta^2 - 1}{2\alpha\beta} =
\Theta(\alpha, \beta),
\]
where~$\alpha = \|x\|$ and~$\beta = \|y\|$. 

If we compute the derivative of~$\Theta(\alpha, \beta)$ with respect
to~$\alpha$ we obtain
\[
\frac{\alpha^2 - \beta^2 + 1}{2 \alpha^2 \beta}.
\]
From this we see that, since~$\varepsilon < \sqrt{2} - 1$, for a
fixed~$\beta \in [1, 1 + \varepsilon)$, function~$\Theta(\alpha,
\beta)$ is increasing in~$\alpha$, for~$\alpha \geq 1$. Similarly, by
taking the derivative with respect to~$\beta$, one may conclude that
for a fixed~$\alpha \in [1, 1 + \varepsilon)$,
function~$\Theta(\alpha, \beta)$ is increasing in~$\beta$, for~$\beta
\geq 1$. So~$\Theta(\alpha, \beta)$ is maximized in our domain when
one takes~$\alpha = \beta = 1 + \varepsilon$. This implies that
\[
\cos \theta \leq \Theta(1 + \varepsilon, 1 + \varepsilon) = 1 -
\frac{1}{2 (1 + \varepsilon)^2},
\]
and so we have~$t(\varepsilon)$.

For the bounds of Table~\ref{tab:improve-ce} we took~$d = 31$.  To
compute upper bounds for~$A(n, \theta)$ we used the semidefinite
programming bound of Bachoc and Vallentin~\cite{BachocV2008}. The
bounds we used for computing Table~\ref{tab:improve-ce} are given in
Table~\ref{tab:eps-meps}.

Finally, we mention that all numerical issues discussed in
Section~\ref{sec:Euclidean space computations} also happen with the
approach we sketched in this section. In particular, the choices of
bases are important for the stability of the semidefinite programming
problems involved. We use the same bases as described in
Section~\ref{sec:Euclidean space computations} though, so we skip a
detailed discussion here. Notice moreover that our bounds are
rigorous, having been checked with the same approach described in
Section~\ref{sec:Euclidean space computations}. 

\def\pz/{\phantom{0}}
\begin{table}[htb]
\begin{tabular}{cl}
{\sl Dimension}&\hfil{\sl $(\varepsilon, U(\varepsilon))$ pairs}\\[2pt]
\vrule height11.5pt width0pt  $3$&$(0.022753, \pz/12)$, $(0.054092, \pz/13)$, $(0.082109, \pz/14)$, $(0.113864, \pz/15)$\\[2pt]
$4$&$(0.008097, \pz/24)$, $(0.017446, \pz/25)$, $(0.025978, \pz/26)$, $(0.036951, \pz/27)$\\[2pt]
$5$&$(0.003013, \pz/45)$, $(0.008097, \pz/46)$, $(0.013259, \pz/47)$, $(0.017446, \pz/48)$\\[2pt]
$6$&$(0.002006, \pz/79)$, $(0.004024, \pz/80)$, $(0.006054, \pz/81)$, $(0.008097, \pz/82)$\\[2pt]
$7$&$(0.001001, 136)$, $(0.002006, 137)$, $(0.003013, 138)$, $(0.004024, 139)$,\\
   &$(0.005037, 140)$\\[2pt]
$9$&$(0.003013, 373)$, $(0.029233, 457)$, $(0.030325, 459)$, $(0.031421, 464)$,\\
   &$(0.032520, 468)$, $(0.033622, 473)$\\[2pt]
\end{tabular}
\bigskip

\caption{For each dimension considered in Table~\ref{tab:improve-ce} we show
  here the sequence~$\varepsilon_1 < \cdots < \varepsilon_m$ and the
  upper bounds~$U(\varepsilon_k)$ used in our application of
  Theorem~\ref{thm:improve-ce}.}
\label{tab:eps-meps}
\end{table}

We refrained from performing similar calculations for higher
dimensions because of two reasons. Firstly, we expect that the
improvements are only minor. Secondly, the computations of the upper
bounds for $M(\varepsilon)$ in higher dimensions require substantially
more time as one needs to solve the semidefinite programs with a high
accuracy solver, see Mittelmann and
Vallentin~\cite{MittelmannVallentin2010}.

\section*{Acknowledgements}

We thank Rudi Pendavingh and Hans D.~Mittelmann for very helpful
discussions from which we learned how to perform the numerically
stable computations.

\end{document}